\documentclass[12pt]{article}
\usepackage{amsfonts}
\usepackage{stmaryrd}
\usepackage{graphicx}
\usepackage{epsfig}
\usepackage{xcolor}
\usepackage{amsmath}
\usepackage{amsthm}
\usepackage{amssymb}
\usepackage{mathrsfs}
\usepackage{float}
\usepackage{multirow}
\usepackage{verbatim}
\usepackage{fancyhdr}
\usepackage{subfigure}
\usepackage{color}
\usepackage{mathtools}
\usepackage{mathrsfs}
\usepackage{sectsty}
\usepackage[title]{appendix}
\usepackage{threeparttable}
\usepackage{dcolumn}
\usepackage{booktabs}
\usepackage{indentfirst}
\usepackage{setspace}
\usepackage{bm}
\usepackage{enumerate}
\usepackage{lineno}
\usepackage{bbm}
\usepackage{tikz}
\usepackage[labelfont=bf,labelsep=period]{caption}
\newcolumntype{d}[1]{D{.}{.}{#1}}
\usepackage{hyperref}
\hypersetup{colorlinks,linkcolor=blue,citecolor=blue}

\newtheorem{theorem}{Theorem}[section]
\newtheorem{example}{Example}[section]

\newtheorem{remark}{Remark}[section]

\allowdisplaybreaks

\begin{document}

\pagenumbering{arabic}
\baselineskip=1.3pc

\vspace*{0.5in}

\begin{center}

{\Large{\bf An energy-based discontinuous Galerkin method for the wave equation with nonsmooth solutions}}

\end{center}

\vspace{.03in}

\centerline{
Yangxin Fu\footnote{School of Mathematical Sciences,
         University of Science and Technology of China,
         Hefei, Anhui 230026, P.R. China.  
         E-mail: yxfu@mail.ustc.edu.cn.},
 Yan Jiang\footnote{School of Mathematical Sciences,
         University of Science and Technology of China, Hefei,
         Anhui 230026, P.R. China.  
         E-mail: jiangy@ustc.edu.cn.
         Research supported by NSFC grant 12271499. }
         and Siyang Wang\footnote{Corresponding author. Department of Mathematics and Mathematical Statistics, Umeå University, Umeå 90187, Sweden. Email: siyang.wang@umu.se
}
}

\vspace{.1in}

\noindent
{\bf Abstract: }
We develop a stable and high-order accurate discontinuous Galerkin method for the second order wave equation, specifically designed to handle nonsmooth solutions. Our approach integrates the energy-based discontinuous Galerkin method with the oscillation-free technique to effectively suppress spurious oscillations near solution discontinuities. Both stability analysis and apriori error estimates are established for common choices of numerical fluxes. We present a series of numerical experiments to confirm the optimal convergence rates for smooth solutions and its robustness in maintaining oscillation-free behavior for nonsmooth solutions in wave equations without or with nonlinear source terms.

\vspace{.1in}


\section{Introduction}

Wave propagation as a ubiquitous phenomenon governing energy transfer, pervades diverse scientific, engineering, and industrial domains. The linear second-order hyperbolic partial differential equation serves as the canonical model for classical wave dynamics, rigorously describing acoustic wave dispersion, elastic medium oscillations, and electromagnetic wave propagation. Beyond linear regimes, nonlinear wave equations emerge as indispensable frameworks for characterizing multiscale interactions from relativistic quantum fields governed by the Klein-Gordon equation to Boussinesq systems modeling nonlinear dispersive waves in coastal hydrodynamics. Such formulations underpin critical applications spanning soliton-mediated energy transport in photonic, seismic inversion for hydrocarbon exploration, and wave-based sensing technologies in defense systems. The inherent complexity of these phenomena, marked by oscillatory singularities and energy cascades, necessitates the development of robust numerical methodologies to ensure computational fidelity in resolving wave interactions.

The discontinuous Galerkin (DG) method is a class of finite element methods that employ discontinuous piecewise polynomial spaces. It was first proposed in 1973 by Reed and Hill to solve the neutron transport equations \cite{reed1973triangular}. Later, the DG method was generalized to different types of equations, including hyperbolic conservation laws \cite{cockburn1999discontinuous, cockburn2001runge}, wave equations \cite{appelo2015new, appelo2020energy, chou2014optimal, grote2006discontinuous, xing2013energy}, and elliptic equations \cite{arnold2002unified}. Due to its inherent advantages such as high-order accuracy, local structure, natural parallelism, and $h$-$p$ adaptivity, the DG method has experienced rapid development. In recent years, DG methods have been demonstrated as effective numerical tools for developing high-order, energy-stable discretizations of time-domain wave propagation problems in complex geometries.

The DG method for the wave equation can be broadly divided into three categories. Firstly, the interior penalty discontinuous Galerkin method (IPDG) \cite{grote2006discontinuous, riviere2003discontinuous, riviere2003discontinuous} discretizes directly in the second-order form. A penalty term is used to ensure coercivity with an appropriately chosen penalty parameter. The second category is the local discontinuous Galerkin method (LDG) \cite{chou2014optimal, xing2013energy, yi2018energy}, where the spatial derivatives are introduced as auxiliary variables. The third category is the energy-based discontinuous Galerkin method (EDG) \cite{appelo2015new, appelo2020energy}, with the time derivative as an auxiliary variable, which is also a natural choice for IPDG with a Runge-Kutta type time discreitzation. The EDG method arises from a general formulation based directly on the Lagrangian form, which is central to the formulation of wave equations in most physical settings. In addition, the EDG method admits a wide variety of mesh-independent energy-conserving or dissipative fluxes. 
 
The aforementioned DG methods are designed primarily for smooth problems. In practical applications, however, initial data may contain discontinuities, or in nonlinear cases, the solution may become discontinuous even when the initial data is smooth. For wave equations with discontinuous solutions, conventional schemes effective for smooth problems often generate spurious oscillations or even fail to converge. To address this challenge, two principal strategies are typically employed. The first type is similar to postprocessing techniques, such as the TVD or TVB limiter \cite{cockburn1989tvb2, cockburn1989tvb, harten1997high, zhong2013simple}. The second type is adding artificial dissipation terms into the scheme. In this direction, the oscillation-free (OF) approach was proposed in \cite{liu2022essentially, liu2024entropy, lu2021oscillation} for first-order hyperbolic equations. The OF approach is a nonlinear scheme with adaptive damping mechanisms using a projection-based methodology such that the damping is small in smooth regions and takes more effect near discontinuities, effectively suppressing oscillations while preserving high-order accuracy in smooth regions \cite{du2023oscillation, liu2022oscillation, peng2025oedg, tao2023oscillation}. Building upon this idea, we  develop a stable and high-order OF-EDG method for the wave equation with nonsmooth solutions, and establish both stability and priori error estimates.

This rest of the paper is organized as follows. In Section \ref{sec:1D}, we begin with a review of the EDG method. Then, we present our proposed formulation of the OF-EDG scheme for one-dimensional problems, and derive an energy estimate to prove that the proposed scheme is stable. In addition, we derive a priori error estimate. In Section \ref{sec:multiD}, we extend the proposed scheme and analysis to multidimensional problems on Cartesian meshes. 
We present a series of numerical examples in Section \ref{sec:num} to verify the theoretical analysis and demonstrate the robustness of the developed method. Finally, concluding remarks are given in Section \ref{sec:con}.

\section{One-dimensional problems}\label{sec:1D}

In this section, we begin by reviewing the EDG method proposed in \cite{appelo2015new} for a one-dimensional model problem. Then we present our new formulation, referred to as the OF-EDG method, which effectively combines the strength of the EDG scheme with the oscillation-free (OF) mechanism. Another key feature of our new formulation is a penalty term that is specially designed to handle piecewise-constant solutions. Finally, we derive an energy estimate and establish a priori error estimate for the proposed method.

\subsection{Review of the EDG scheme}

Let us consider the wave equation in one space dimension,
\begin{equation}\label{eq:wave_1D}
    u_{tt} = u_{xx},\quad (x,t)\in (a,b)\times (0,T],
\end{equation}
with initial conditions
$$u(x,0) = u_0(x), \quad u_t(x,0) = u_1(x).$$ 
Periodic boundary conditions are considered in this work, however, this choice does not constitutes a limitation of the method, which readily accommodates other boundary conditions, e.g., Dirichlet, Neumann and characteristic boundary conditions.

We discretize the computation domain $[a,b]$ by a mesh consisting of cells
\begin{equation}
    I_j = [x_{j-1/2},x_{j+1/2}],\qquad 1\leq j\leq N,
\end{equation}
with
\begin{equation}
    a = x_{1/2}< x_{3/2}<\cdots<x_{N+1/2} = b,
\end{equation}
and denote
\begin{equation}
\begin{aligned}
    & x_j = \dfrac{1}{2}(x_{j-1/2} + x_{j+1/2}), \quad 1\leq j\leq N,\\
    & h = \max_{1\leq j\leq N}h_j, \quad h_j = x_{j+1/2} - x_{j-1/2}, \quad 1\leq j\leq N.
\end{aligned}
\end{equation}
We also assume that the mesh is quasi-uniform, that is, there exists a constant $\rho > 0$ such that for all $j$ hold $\rho h\leq h_j$ as $h$ goes to zero.

Associated with the meshes, we define the discontinuous finite element space as follows,
\begin{equation}
    V_h^k = \{v\in L^2([a,b]):v|_{I_j}\in P^k(I_j),j = 1,2,...,N\},
\end{equation}
where $P^k(I_j)$ is the space of polynomials of degree at most $k$ on $I_j$. To facilitate the DG formulation, we denote the jump of $w_h$ at $x_{j+1/2}$ as 
$$[\![w_h]\!]_{j+1/2} = w_h|^+_{j+1/2} - w_h|_{j+1/2}^-,$$  
where $w_h|_{j+1/2}^{\pm} = \lim_{\epsilon\to 0^+} w_h(x_{j+1/2}\pm \epsilon)$ represents the left or right limit of $w_h$ at $x_{j+1/2}$.

The EDG method seeks approximation to a corresponding system by introducing the time derivative as a new variable $v = u_t$, yielding
\begin{equation}\label{eq:wave_sys_1D}
\begin{cases}
    u_t = v,\\
    v_t = u_{xx},\\
\end{cases} \quad (x,t)\in (a,b)\times (0,T].
\end{equation}
We choose the numerical solution space $u_h\in V_h^p$ and $v_h\in V_h^q$. 
Next, we test equation \eqref{eq:wave_sys_1D} by $\phi^{(u)}_{xx}$ and $\phi^{(v)}$, respectively, where $\phi^{(u)}\in V_h^p$ and $\phi^{(v)}\in V_h^q$. 
After integrating in space and using the integration by parts formula, we obtain 
\begin{equation}\label{eq:ESDG}
\begin{cases}
    \displaystyle
    \int_{I_j} \left( (u_h)_t - v_h \right)_x \phi^{(u)}_{x} dx 
    = \left(\widehat{v}_{j+1/2} - v_h|^{-}_{j+1/2} \right) \phi_x^{(u)}|^{-}_{j+1/2} 
    -
    \left( \widehat{v}_{j-1/2} - v_h|^{+}_{j-1/2} \right) \phi_x^{(u)}|^{+}_{j-1/2}, \\
    \displaystyle \int_{I_j} \left( (v_h)_t \phi^{(v)} + (u_h)_x \phi^{(v)}_{x} \right) dx 
    = \widehat{u_x}|_{j+1/2}\phi^{(v)}|^{-}_{j+1/2} - 
    \widehat{u_x}|_{j-1/2} \phi^{(v)}|^{+}_{j-1/2},
\end{cases}
\end{equation}
where $\widehat{v}_{j+1/2},\widehat{u_x}|_{j+1/2}$ are the numerical fluxes defined on cell interfaces. The numerical fluxes take the general form,
\begin{equation}\label{eq:numerical_flux}
\begin{aligned}
    \widehat{v} =&  \alpha v_h^+ + (1-\alpha)v_h^-+\tau[\![(u_h)_x]\!] 
    = \frac{v_h^+ + v_h^-}{2} - (\frac{1}{2}-\alpha) [\![v_h]\!]  +\tau[\![(u_h)_x]\!],\\
    \widehat{u_x} =&  (1-\alpha)(u_h)_x^+ +\alpha (u_h)_x^- +\beta[\![v_h]\!]
    = \frac{(u_h)_x^+ + (u_h)_x^-}{2} + (\frac{1}{2}-\alpha) [\![(u_h)_x ]\!] +\beta[\![v_h]\!],
\end{aligned}
\end{equation}
where $\alpha\in [0,1]$ and $\tau \geq 0,\ \beta\geq 0$.
Special choices of the parameters are given:
\begin{align}
\begin{array}{lll}
    & \text{Central flux (C-flux)}: & \alpha = \dfrac{1}{2}, \quad \beta = \tau = 0.\\
    & \text{Alternating flux (A-flux)}:& \alpha = 0 \,\, \text{or} \,\, 1, \quad \beta = \tau = 0.\\
    & \text{Sommerfeld flux (S-flux)}: & \alpha = \dfrac{1}{2}, \quad \beta = \dfrac{1}{2s}, \quad \tau = \dfrac{s}{2},\,\, s>0.
\end{array}
\end{align}

We note that for constant test function $\phi^{(u)}$, the first equation in \eqref{eq:ESDG} is reduced to a trivial relation $0=0$. Consequently, we need to complement \eqref{eq:ESDG} by 
$\int_{I_j} (u_h)_t dx = \int_{I_j} v_h dx$.
The EDG scheme can then be stated as follows: find $u_h\in V_h^p,v_h\in V_h^q$ such that the following weak formulation holds for any test function $\phi^{(u)}\in V_h^p$ and $\phi^{(v)}\in V_h^q$,  $j=1,2,\ldots,N$,
\begin{equation}\label{eq:ESDG_full}
\begin{cases}
    \displaystyle
    \int_{I_j} \left( (u_h)_t - v_h \right) dx =0, \\
    \displaystyle
    \int_{I_j} \left( (u_h)_t - v_h \right)_x \phi^{(u)}_{x} dx 
    = \left(\widehat{v}_{j+1/2} - v_h|^{-}_{j+1/2} \right) \phi_x^{(u)}|^{-}_{j+1/2} 
    -\left( \widehat{v}_{j-1/2} - v_h|^{+}_{j-1/2} \right) \phi_x^{(u)}|^{+}_{j-1/2}, \\
    \displaystyle \int_{I_j} \left( (v_h)_t \phi^{(v)} + (u_h)_x \phi^{(v)}_{x} \right) dx
    = \widehat{u_x}|_{j+1/2}\,\phi^{(v)}|^{-}_{j+1/2} - 
    \widehat{u_x}|_{j-1/2} \, \phi^{(v)}|^{+}_{j-1/2}.
\end{cases}
\end{equation}

The EDG scheme can be easily generalized to the wave equation in multi-dimensions with general boundary conditions \cite{appelo2015new}. The technique has also been applied to the semilinear wave equation \cite{appelo2020energy}, the elastic wave equation \cite{appelo2018energy}, and the acoustic-elasto system \cite{appelo2019energy}. These schemes possess good properties such as energy stability and optimal convergences.  However, for problems with discontinuous solutions, spurious oscillations occur near the discontinuities. To overcome this challenge, we design a new scheme that maintains high order accuracy in the smooth region, and captures discontinuities without oscillation.

\subsection{The OF-EDG scheme}
Now we proceed to design a DG scheme that poses high-order accuracy and can control spurious oscillations automatically. In particular, we follow the idea of the OFDG scheme \cite{lu2021oscillation} with a damping term in each cell to avoid spurious oscillations. However, we have found that the scheme may give the wrong location of discontinuities for piecewise-constant solutions. To address this issue, we add a penalty term inspired by the IPDG scheme \cite{grote2006discontinuous}.

The new semi-discrete DG scheme is defined as follows:
find $u_h\in V_h^p$ and $v_h\in V_h^{q}$ such that 
\begin{equation}\label{eq:DG_1D}
\begin{cases}
    \displaystyle
    {\int_{I_j} \left( (u_h)_t - v_h \right) dx }&=0, \\
    \displaystyle
    \int_{I_j} \left( (u_h)_t - v_h \right)_x \phi^{(u)}_{x} dx 
    &= \left(\widehat{v}_{j+1/2} - v_h|^{-}_{j+1/2} \right) \phi_x^{(u)}|^{-}_{j+1/2} -
    \left( \widehat{v}_{j-1/2} - v_h|^{+}_{j-1/2} \right) \phi_x^{(u)}|^{+}_{j-1/2} \\
    &\quad
    +\dfrac{c}{h^2} \left( [\![u_h]\!]_{j+1/2} \phi^{(u)}|^-_{j+1/2} -[\![u_h]\!]_{j-1/2} \phi^{(u)}|^+_{j-1/2} \right)\\
    &\quad{\displaystyle  - \sum_{l=1}^{p} \frac{\sigma_j^l}{h_j} \int_{I_j} \left((u_h)_x - \mathbb{P}^{l-1}(u_h)_x\right) \phi^{(u)}_x dx}, 
    \qquad \forall \, {\phi^{(u)}\in V_h^p(I_j)}, \\
    \displaystyle \int_{I_j} \left( (v_h)_t \phi^{(v)} + (u_h)_x \phi^{(v)}_{x} \right) dx 
    &= \widehat{u_x}|_{j+1/2}\phi^{(v)}|^{-}_{j+1/2} - 
    \widehat{u_x}|_{j-1/2} \phi^{(v)}|^{+}_{j-1/2}\\
    &\quad {\displaystyle - \sum_{l=0}^{q} \dfrac{\tilde{\sigma}_j^l}{h_j} \int_{I_j} \left(v_h - \mathbb{P}^{l-1}v_h \right) \phi^{(v)} dx},
    \qquad \forall \, {\phi^{(v)}\in V_h^q(I_j)}.
\end{cases}
\end{equation}

\noindent 
Here, $c>0$ is  the penalty parameter and the corresponding term is referred to as the penalty term. In the OF damping terms, the operator $\mathbb{P}^l, l\geq 0$, is the standard local $L^2$ projection, i.e., for any function $w$, find $\mathbb{P}^lw \in  V^l_h$ such that
\begin{equation}
    \int_{I_j}(\mathbb{P}^lw - w)\phi(x)dx = 0, \quad \forall \, \phi(x) \in V_h^l(I_j).
\end{equation}
In addition, we use the convention $\mathbb{P}^{-1} = \mathbb{P}^0$. 
The damping parameters $\sigma_j^l \geq 0, \tilde{\sigma}_j^l \geq 0$ control dissipation, and take the form: 
\begin{equation}
    \begin{aligned}
        \sigma^l_j =&  \dfrac{2(2l+1)}{(2p-1)}\dfrac{h_j^l}{l!}([\![\partial_x^lu_h]\!]^2_{j+1/2}+[\![\partial_x^lu_h]\!]^2_{j-1/2})^{1/2}, \quad l\geq 1, \\
        \tilde{\sigma}^l_j =&  \dfrac{2(2l+1)}{(2q-1)}\dfrac{h_j^{l+1}}{l!}([\![\partial_x^lv_h]\!]^2_{j+1/2}+[\![\partial_x^lv_h]\!]^2_{j-1/2})^{1/2}, \quad l\geq 0.\\
    \end{aligned}
\end{equation}

We emphasize that both the damping terms and the penalty term play crucial roles in handling discontinuous solutions.  The damping terms are designed to be negligible in regions where the solution is smooth, but take effect in regions when the solution is less regular, effectively suppressing spurious oscillations. Since the original EDG fluxes and the damping terms depend only on the derivatives of the numerical solution, they have no influence when the solution is constant within each cell. This scenario occurs, for example, when the initial data is piecewise constant with discontinuities located on the cell interfaces. In such cases,  the penalty term becomes essential because it is based on the jump in the numerical solution $u_h$. The $h$-scaling in the coefficients of the penalty term and the damping terms is carefully chosen so that these terms do not degrade convergence rates for problems with smooth solutions, while still are strong enough to suppress oscillations in regions where the solution lacks smoothness. In Section \ref{sec:num}, we demonstrate in numerical tests that both the damping terms and penalty term are necessary for discontinuous solution,

\subsection{Stability analysis}

In the following, we derive stability analysis for the semi-discrete OF-EDG scheme \eqref{eq:DG_1D} with general numerical fluxes \eqref{eq:numerical_flux}.

\begin{theorem}[Semi-discrete stability]
    The OF-EDG scheme \eqref{eq:DG_1D} with general numerical flux \eqref{eq:numerical_flux} satisfies
    \begin{equation}\label{eq:ES_1D}
    \dfrac{d \, \mathbb{E}_h}{dt}\leq 0
    \end{equation} 
    for any positive integers $p$ and $q$, and parameters $c,\tau,\beta\geq 0$, where the discrete energy is defined as  
    \begin{equation} \label{eq:energy_1D}
        \mathbb{E}_h = \int_{\Omega}((u_h)_x^2 +(v_h)^2)dx.
    \end{equation}
\end{theorem}

\begin{proof}
    Taking test function $(\phi^{(u)},\phi^{(v)}) = (u_h,v_h)$ in \eqref{eq:DG_1D} and summing them, we  obtain
    \begin{equation}\label{eq:proof_add1}
    \begin{aligned}
        \int_{I_j} (u_h)_{xt} (u_h)_x + (v_h)_t v_h dx 
    = & \int_{I_j} (v_h)_x (u_h)_{x} dx-\int_{I_j} (u_h)_x(v_h)_xdx \\
     &+ \widehat{u_x}|_{j+1/2}v_h|^{-}_{j+1/2} + \left(\widehat{v}_{j+1/2} - v_h|^{-}_{j+1/2} \right) (u_h)_x|^{-}_{j+1/2} \\
     &- \widehat{u_x}|_{j-1/2}v_h|^{+}_{j-1/2} - \left(\widehat{v}_{j-1/2} - v_h|^{+}_{j-1/2} \right) (u_h)_x|^{+}_{j-1/2} \\
     &+\dfrac{c}{h^2}\left( [\![u_h]\!]|_{j+1/2} u_h|_{j+1/2}^- - [\![u_h]\!]|_{j-1/2}u_h|_{j-1/2}^+ \right)\\
     &- \sum_{l=1}^{p} \frac{\sigma_j^l}{h_j} \int_{I_j} \left((u_h)_x - \mathbb{P}^{l-1}(u_h)_x\right)(u_h)_x dx \\
     &-  \sum_{l=0}^{q} \frac{\tilde{\sigma}_j^l}{h_j} \int_{I_j} \left(v_h - \mathbb{P}^{l-1}v_h \right) v_h dx. 
    \end{aligned}
    \end{equation}

    \noindent
    Substituting the general flux \eqref{eq:numerical_flux} into the second term and third term on the right hand side yields
    \begin{align*}
         \widehat{u_x}|_{j+1/2}v_h|^{-}_{j+1/2} +& \left(\widehat{v}_{j+1/2} - v_h|^{-}_{j+1/2} \right) (u_h)_x|^{-}_{j+1/2} 
        - \widehat{u_x}|_{j-1/2}v_h|^{+}_{j-1/2} - \left(\widehat{v}_{j-1/2} - v_h|^{+}_{j-1/2} \right) (u_h)_x|^{+}_{j-1/2}  \\
        =& \left( (1-\alpha)(u_h)_x^++\alpha(u_h)_x^-+\beta[\![v_h]\!] \right)|_{j+1/2} \, v_h|^{-}_{j+1/2} \nonumber\\
    &+ \left((\alpha v_h^+ +(1-\alpha) v_h^- +\tau[\![(u_h)_x]\!])|_{j+1/2} -v_h|^{-}_{j+1/2} \right) (u_h)_x|^{-}_{j+1/2} \nonumber\\
    &- \left( (1-\alpha)(u_h)_x^++\alpha(u_h)_x^-+\beta[\![v_h]\!] \right)|_{j-1/2} \, v_h|^{+}_{j-1/2} \nonumber\\
    &- \left((\alpha v_h^++(1-\alpha) v_h^-+\tau[\![(u_h)_x]\!])|_{j-1/2} - v_h|^{+}_{j-1/2} \right) (u_h)_x|^{+}_{j-1/2}\nonumber\\
    =&\ \alpha v_h|_{j+1/2}^+(u_h)_x|_{j+1/2}^-
    -\alpha v_h|_{j-1/2}^+(u_h)_x|_{j-1/2}^- \\
    &+(1-\alpha)v_h|_{j+1/2}^-(u_h)_x|_{j+1/2}^+ 
    -(1-\alpha)v_h|_{j-1/2}^-(u_h)_x|_{j-1/2}^+\\
    & +\tau[\![(u_h)_x]\!]_{j+1/2} (u_h)_x|_{j+1/2}^-
    -\tau[\![(u_h)_x ]\!]_{j-1/2} (u_h)_x|_{j-1/2}^+ \\
    &+\beta[\![v_h]\!]_{j+1/2} \, v_h|_{j+1/2}^-
    -\beta[\![v_h]\!]_{j-1/2}\, v_h|_{j-1/2}^+ . 
    \end{align*}

\noindent    
Summing \eqref{eq:proof_add1} over $j$ with periodic boundary condition yields
\begin{equation*}
\begin{aligned}
    & \frac{1}{2}\frac{d}{dt} \sum_{j} \int_{I_j} 
    \left( (u_h)_x^2 + (v_h)^2\right) dx \\
    =&{-\tau\sum_{j}[\![(u_h)_x]\!]_{j+1/2}^2-\beta\sum_{j}[\![v_h]\!]_{j+1/2}^2}
    -\sum_{j}\dfrac{c}{h^2}[\![u_h]\!]_{j+1/2}^2 \nonumber\\
    &-\sum_{j} \sum_{l=1}^{p} 
    \dfrac{\sigma_j^l}{h_j} \int_{I_j} \left((u_h)_x -\mathbb{P}^{l-1}(u_h)_x \right)^2 dx
    - \sum_{j} \sum_{l=0}^{q} \frac{\tilde{\sigma}_j^l}{h_j} \int_{I_j} \left(v_h - \mathbb{P}^{l-1}v_h \right)^2 dx . 
\end{aligned}
\end{equation*}
The desired energy estimate is obtained when $c \geq 0$, $\tau \geq0$, $\beta\geq0$.
\end{proof}

\begin{remark}
   Without the penalty term and the damping terms, the scheme with A-flux or C-flux is energy-conserving since $\tau = \beta = 0$, meanwhile the scheme with S-flux is energy dissipative; after adding the penalty term and the damping terms, the scheme always dissipates energy.
\end{remark}

\subsection{A priori error estimate}

We establish a priori error estimates in the energy norm of the OF-EDG scheme \eqref{eq:DG_1D}.

\begin{theorem}[Error estimate]\label{Thm_err}
    Assume the exact solution $u\in H^{p+1}(\Omega)$, $v\in H^{q+1}(\Omega)$, and $p-2\leq q\leq p$, then   
    \begin{equation}\label{eq:error_1D}
        \Vert(u-u_h)_x\Vert^2_{L^2(\Omega)} + \Vert v-v_h\Vert_{L^2(\Omega)}^2 \lesssim h^{2\gamma}\left( |u|^2_{H^{p+1}(\Omega)} +|v|^2_{H^{q+1}(\Omega)} \right),
    \end{equation}
    where $\gamma=\min(p', q')$, and
    \begin{equation}
    p' = 
    \begin{cases} 
    p-1, & \tau=0 , \\
    p-1/2, &\tau> 0,
    \end{cases} \qquad
    q' = 
    \begin{cases} 
    q, & \beta=0, \\
    q+1/2, &\beta> 0.
    \end{cases}
    \end{equation}
    Here, the notation $A \lesssim B$ means that there exists a constant $C_0>0$ independent of $h$ such that $A\leq C_0\,B$.
\end{theorem}

\begin{proof}
We denote $V^{p,q}=V_h^p\times V_h^q$ to be the space of numerical solution $U_h = (u_h,v_h)$,  
$U = (u,v)$ is the exact solution, and $\Phi = (\phi^{(u)},\phi^{(v)})\in V^{p,q}$ is the test function. We introduce the notations
\begin{equation*}
\begin{aligned}
    B_j(U_h, \Phi) =& \int_{I_j} \left( ((u_h)_t - v_h)_x\phi^{(u)}_x + (v_h)_t\phi^{(v)} + (u_h)_x\phi^{(v)}_x \right) dx  \\
    &- \left(\widehat{v}_{j+1/2} - v_h|^{-}_{j+1/2} \right) \phi_x^{(u)}|^{-}_{j+1/2} +\left( \widehat{v}_{j-1/2} - v_h|^{+}_{j-1/2} \right) \phi_x^{(u)}|^{+}_{j-1/2}  \\
    &- \widehat{u_x}|_{j+1/2}\phi^{(v)}|^{-}_{j+1/2} + 
    \widehat{u_x}|_{j-1/2} \phi^{(v)}|^{+}_{j-1/2}, \\
    B(U_h,\Phi) =& \sum_{j} B_j(U_h,\Phi).
\end{aligned}
\end{equation*}
We also define 
\begin{equation*}
\begin{aligned}
    D_{u}(u_h, \phi^{(u)}) = \sum_j D_{u,j}(u_h, \phi^{(u)}), \quad 
    D_{u,j}(u_h, \phi^{(u)}) = & \sum_{l=1}^{p} \frac{\sigma_j^l}{h_j} \int_{I_j} \left((u_h)_x - \mathbb{P}^{l-1}(u_h)_x\right)\phi^{(u)}_x dx,\\
     D_{v}(v_h, \phi^{(v)}) = \sum_j D_{v,j}(v_h, \phi^{(v)}), \quad 
     D_{v,j}(v_h, \phi^{(v)}) = & \sum_{l=0}^{q} \frac{\tilde{\sigma}_j^l}{h_j} \int_{I_j} \left(v_h - \mathbb{P}^{l-1}v_h \right) \phi^{(v)} dx,\\
     P_{u}(u_h,\phi^{(u)}) = \sum_j P_{u,j}(u_h,\phi^{(u)}), \quad 
     P_{u,j}(u_h,\phi^{(u)})=&- \dfrac{c}{h^2}([\![u_h]\!]|_{j+1/2}\phi^{(u)}|_{j+1/2}^--[\![u_h]\!]|_{j-1/2}\phi^{(u)}|_{j-1/2}^+).
\end{aligned}
\end{equation*}
Thus, for any $\Phi\in V^{p,q}$, the exact solution $U$ satisfies $B_j(U,\Phi)=0$, and the numerical solution $U_h$ satisfies
\begin{equation*}  
    B_j(U_h,\Phi) +D_{u,j}(u_h,\phi^{(u)}) +D_{v,j}(v_h,\phi^{(v)}) +P_{u,j}(u_h,\phi^{(u)})=0,\quad \forall \,\,\Phi\in V^{p,q}.
\end{equation*} 
Let the error between the exact and numerical solution be 
\begin{align*}
    e_u = u-u_h,\quad e_v=v-v_h,\quad D_h =(e_u,e_v),
\end{align*}
We have the error equation,
\begin{equation}\label{eq:error_eqn}            
    B_j(D_h,\Phi) = D_{u,j}(u_h,\phi^{(u)}) +D_{v,j}(v_h,\phi^{(v)}) +P_{u,j}(u_h,\phi^{(u)}).
\end{equation}
Next, we define the differences
\begin{equation*}
\begin{aligned}
    & \tilde{e}_u=\tilde{u}_h-u_h,\quad \tilde{e}_v=\tilde{v}_h-v_h,\quad \tilde{D}_h=(\tilde{e}_u,\tilde{e}_v)\in V^{p,q},\\
    & \delta_u=\tilde{u}_h-u,\quad\,\,\, \delta_v = \tilde{v}_h-v,\quad\,\,\, \Delta_h = (\delta_u,\delta_v),
\end{aligned}
\end{equation*}
where $\tilde{v}_h\in V_h^q$ is the $L^2$ projection of $v$, i.e., $\tilde{v}_h=\mathbb{P}^q v$,
and $\tilde{u}_h\in V^p_h$ is obtained via a Ritz-type projection 
\begin{equation*}
\begin{cases}
    \displaystyle
    \int_{I_j} (u-\tilde{u}_h)_x \phi^{(u)}_x dx = 0 , \quad \forall \,\phi^{(u)}\in V^{p}_h,\\
    \displaystyle
    \int_{I_j} (u - \tilde{u}_h)dx = 0.
\end{cases}
\end{equation*}
Then, for $u\in H^{p+1}(\Omega)$ and $v\in H^{q+1}(\Omega)$, we have the basic results following from the Bramble-Hilbert lemma \cite{ciarlet2002finite}  
\begin{equation}\label{eq:prop_add1}
\begin{aligned}
    \Vert\delta_u\Vert_{L^2(I_j)} \lesssim h^{p+1}|u|_{H^{p+1}(I_j)},\quad
    \Vert \delta_v\Vert_{L^2(I_j)}\lesssim h^{q+1}|v|_{H^{q+1}(I_j)}.
\end{aligned}
\end{equation}
More generally, we have the following properties \cite{ciarlet2002finite}
\begin{equation}\label{Ritz_Projection}
    \begin{aligned}
        \sum_j \left( h^l\Vert\partial_x^l\delta_u\Vert_{L^2(I_j)} + h^{l+1/2}\Vert\partial_x^l\delta_u\Vert_{L^2(\partial I_j)} \right) \lesssim h^{p+1}|u|_{H^{p+1}(\Omega)},\quad l = 0,\ldots,p,\\
        \sum_j \left( h^l\Vert\partial_x^l\delta_v\Vert_{L^2(I_j)} + h^{l+1/2}\Vert\partial_x^l\delta_v\Vert_{L^2(\partial I_j)}\right) \lesssim h^{q+1}|v|_{H^{q+1}(\Omega)},\quad l = 0,\ldots,q,
    \end{aligned}
\end{equation}
and 
\begin{equation}\label{eq:error_flux}
    \begin{aligned}
        \sum_j \left(\Vert\widehat{(\delta_u)_x}\Vert_{L^2(\partial I_j)} + \Vert\widehat{\delta_v}\Vert_{L^2(\partial I_j)}\right) \lesssim h^{p-1/2}|u|_{H^{p+1}(\Omega)} + h^{q+1/2}|v|_{H^{q+1}(\Omega)}.
    \end{aligned}
\end{equation}
In addition, we also need inverse inequalities, for any $w\in V_h^k$, there exists a positive constant independent of $w$ and $h$ such that
\begin{equation}\label{eq:inverse}
    \Vert w_x\Vert_{L^2(I_j)}\lesssim h^{-1}\Vert w\Vert_{L^2(I_j)},\quad 
    \Vert w \Vert_{\partial I_j}\lesssim h^{-1/2}\Vert w\Vert_{L^2(I_j)}, 
    \quad \, \forall \, w(x)\in V_h^k,
\end{equation}
Clearly, we have $D_h = \tilde{D}_h-\Delta_h$, thus the error equation \eqref{eq:error_eqn} can be rewritten as
\begin{equation*}
B_j(\tilde{D}_h, \Phi) = B_j(\Delta_h, \Phi) +D_{u,j}(u_h,\phi^{(u)}) +D_{v,j}(v_h,\phi^{(v)}) +P_{u,j}(u_h,\phi^{(u)}).
\end{equation*}
Now, we choose $\Phi = \tilde{D}_h$ and sum over $j$, and obtain
\begin{equation}\label{eqn:err_rhs}
\begin{aligned}
    \dfrac{d}{dt} \dfrac{1}{2}\sum_j\int_{I_j}\left( (\tilde{e}_u)_x^2 + (\tilde{e}_v)^2 \right) dx
    =&B(\Delta_h,\tilde{D}_h) -\sum_j \left( \beta[\![\tilde{e}_v]\!]_{j+1/2}^2+\tau [\![(\tilde{e}_u)_x]\!]_{j+1/2}^2 
     \right)\\
    &+D_u(u_h,\tilde{e}_u) +D_v(v_h,\tilde{e}_v) +P_u(u_h,\tilde{e}_u).
\end{aligned}
\end{equation}

\noindent
In the following, we derive estimates of the right-hand side term by term.  
For notation clarity, we introduce 
\begin{equation*}
    S(\tilde{e}_u,\tilde{e}_v) = \sum_j \left( \beta[\![\tilde{e}_v]\!]_{j+1/2}^2+\tau [\![(\tilde{e}_u)_x]\!]_{j+1/2}^2 
     \right).
\end{equation*}

The first two terms on the right-hand side of \eqref{eqn:err_rhs} can be estimated by following the error analysis in \cite{appelo2015new}. To start, we use the definitions of $\tilde{u}_h$ and $\tilde{v}_h$ and integration by parts to obtain
\begin{equation*}
    \begin{aligned}
    B(\Delta_h, \tilde{D}_h) 
    =&\sum_j\int_{I_j}\left( ((\delta_u)_t - \delta_v)_x(\tilde{e}_u)_x + (\delta_v)_t\tilde{e}_v + (\delta_u)_x(\tilde{e}_v)_x \right) dx\\
    & + \sum_j \left( \widehat{\delta_v}|_{j+1/2} [\![(\tilde{e}_u)_x]\!]_{j+1/2}- [\![\delta_v (\tilde{e}_u)_x]\!]_{j+1/2} + \widehat{(\delta_u)_x}|_{j+1/2} [\![\tilde{e}_v]\!]_{j+1/2} \right) \\
    = &\sum_j\int_{I_j} \left( -(\delta_v)_x(\tilde{e}_u)_x + (\delta_u)_x(\tilde{e}_v)_x \right)dx  \\
    & + \sum_j \left( \widehat{\delta_v}|_{j+1/2} [\![(\tilde{e}_u)_x]\!]_{j+1/2}- [\![\delta_v (\tilde{e}_u)_x]\!]_{j+1/2} + \widehat{(\delta_u)_x}|_{j+1/2} [\![\tilde{e}_v]\!]_{j+1/2} \right).\\
    \end{aligned}
\end{equation*}
Since $p-2\leq q\leq p$, then $\tilde{e}_{xx}\in V^q_h$ and $\tilde{e}_v\in V_h^p$, we get
\begin{equation*}
    \begin{aligned}
    B(\Delta_h, \tilde{D}_h)=&\sum_j\int_{I_j} \left( \delta_v(\tilde{e}_u)_{xx} + (\delta_u)_x(\tilde{e}_v)_x \right) dx \\
    & +    \sum_j \left( \widehat{\delta_v}|_{j+1/2}[\![(\tilde{e}_u)_x]\!]_{j+1/2} 
    +\widehat{(\delta_u)_x}|_{j+1/2} [\![\tilde{e}_v]\!]_{j+1/2} \right) \\
    = &\sum_j \left( \widehat{\delta_v}|_{j+1/2}[\![(\tilde{e}_u)_x]\!]_{j+1/2} 
    +\widehat{(\delta_u)_x}|_{j+1/2} [\![\tilde{e}_v]\!]_{j+1/2} \right).
    \end{aligned}
\end{equation*}

Now for the term $B(\Delta_h,\tilde{D}_h) - S(\tilde{e}_u,\tilde{e}_v)$, we  consider the following four cases:\\
Case 1: $\tau = 0$, $\beta = 0$,\\
\begin{equation*}\label{case1}
    \begin{aligned}
        B(\Delta_h,\tilde{D}_h) - S(\tilde{e}_u,\tilde{e}_v)
        \lesssim& \sum_j \left( \Vert\widehat{\delta_v}\Vert_{L^2(\partial I_j)}\Vert(\tilde{e}_u)_x\Vert_{L^2(\partial I_j)} + \Vert\widehat{(\delta_u)_x}\Vert_{L^2(\partial I_j)}\Vert\tilde{e}_v\Vert_{L^2(\partial I_j)} \right)  \\
        \lesssim& \sum_j \left( h^{-1/2}\Vert\widehat{\delta_v}\Vert_{L^2(\partial I_j)}\Vert(\tilde{e}_u)_x\Vert_{L^2(I_j)} + h^{-1/2}\Vert\widehat{(\delta_u)_x}\Vert_{L^2(\partial I_j)}\Vert\tilde{e}_v\Vert_{L^2(I_j)} \right)  \\
        \lesssim&
        \sum_j\left( h^{-1}\Vert\widehat{\delta_v}\Vert^2_{L^2(\partial I_j)} + \Vert(\tilde{e}_u)_x)\Vert^2_{L^2(I_j)} + h^{-1}\Vert\widehat{(\delta_u)_x}\Vert^2_{L^2(\partial I_j)} + \Vert\tilde{e}_v\Vert_{L^2(I_j)}^2 \right) \\
        \lesssim&  \Vert(\tilde{e}_u)_x)\Vert^2_{L^2(\Omega)} + \Vert\tilde{e}_v\Vert_{L^2(\Omega)}^2  + h^{2(p-1)}|u|^2_{H^{p+1}(\Omega)} + h^{2q}|v|_{H^{q+1}(\Omega)}^2 .
    \end{aligned}
\end{equation*}

\noindent
Case 2: $\tau > 0$, $\beta = 0$, there exists a constant $\epsilon>0$ such that $\dfrac{\epsilon}{2}\leq \tau$, then we have,
\begin{equation*}\label{case2}
    \begin{aligned}
        B(\Delta_h,\tilde{D}_h) - S(\tilde{e}_u,\tilde{e}_v)
     =&\sum_j \widehat{(\delta_u)_x}|_{j+1/2} [\![\tilde{e}_v]\!]_{j+1/2}+\widehat{\delta_v}|_{j+1/2} [\![(\tilde{e}_u)_x]\!]_{j+1/2}-\tau [\![(\tilde{e}_u)_x]\!]_{j+1/2}^2\\ 
     \lesssim& \sum_j \Vert\widehat{(\delta_u)_x}\Vert_{L^2(\partial I_j)}\Vert\tilde{e}_v\Vert_{L^2(\partial I_j)} + \dfrac{1}{2\epsilon} \Vert\widehat{\delta_v}\Vert^2_{L^2(\partial I_j)} + \frac{\epsilon}{2} [\![(\tilde{e}_u)_x]\!]_{j+1/2}^2-\tau [\![(\tilde{e}_u)_x]\!]_{j+1/2}^2   \\
     \lesssim& \sum_j h^{-1/2}\Vert\widehat{(\delta_u)_x}\Vert_{L^2(\partial I_j)}\Vert\tilde{e}_v\Vert_{L^2(I_j)} + \dfrac{1}{2\epsilon}\Vert\widehat{\delta_v}\Vert^2_{L^2(\partial I_j)}\\
     \lesssim& \Vert\tilde{e}_v\Vert^2_{L^2(\Omega)} + \sum_j\left( h^{-1}\Vert\widehat{(\delta_u)_x}\Vert_{L^2(\partial I_j)}^2 + \dfrac{1}{\epsilon}\Vert\widehat{\delta_v}\Vert^2_{L^2(\partial I_j)}\right) \\
     \lesssim& \Vert\tilde{e}_v\Vert^2_{L^2(\Omega)} + h^{2(p-1)}|u|^2_{H^{p+1}(\Omega)} + h^{2q+1}|v|^2_{H^{q+1}(\Omega)}. 
    \end{aligned}
\end{equation*}

\noindent
Case 3: $\tau = 0$, $\beta > 0$, there exists a constant $\epsilon>0$ such that $\dfrac{\epsilon}{2}\leq \beta$,
\begin{equation*}\label{case3}
    \begin{aligned}
        B(\Delta_h,\tilde{D}_h) - S(\tilde{e}_u,\tilde{e}_v)
     =&\sum_j \widehat{(\delta_u)_x}|_{j+1/2} [\![\tilde{e}_v]\!]_{j+1/2}+\widehat{\delta_v}|_{j+1/2} [\![(\tilde{e}_u)_x]\!]_{j+1/2}-\beta[\![\tilde{e}_v]\!]_{j+1/2}^2\\ 
     \lesssim&\sum_j \Vert\widehat{\delta_v} \Vert_{L^2(\partial I_j)}\Vert(\tilde{e}_u)_x\Vert_{L^2(\partial I_j)} 
     + \dfrac{1}{2\epsilon} \Vert\widehat{(\delta_u)_x}\Vert^2_{L^2(\partial I_j)} 
     + \frac{\epsilon}{2} [\![\tilde{e}_v]\!]_{j+1/2}^2 - \beta[\![\tilde{e}_v]\!]_{j+1/2}^2 \\
     \lesssim& \sum_j h^{-1/2}\Vert\widehat{\delta_v} \Vert_{L^2(\partial I_j)}\Vert(\tilde{e}_u)_x\Vert_{L^2(I_j)} + \dfrac{1}{2\epsilon} \Vert\widehat{(\delta_u)_x}\Vert^2_{L^2(\partial I_j)}\\
     \lesssim& \Vert(\tilde{e}_u)_x\Vert_{L^2(\Omega)}^2 + \sum_j\left(h^{-1}\Vert\widehat{\delta_v} \Vert_{L^2(\partial I_j)}^2 + \dfrac{1}{\epsilon} \Vert\widehat{(\delta_u)_x}\Vert^2_{L^2(\partial I_j)}\right)\\
     \lesssim & \Vert(\tilde{e}_u)_x\Vert_{L^2(\Omega)}^2 + h^{2p-1}|u|^2_{H^{p+1}(\Omega)} + h^{2q}|v|^2_{H^{q+1}(\Omega)}.
    \end{aligned}
\end{equation*}

\noindent
Case 4: $\tau > 0$, $\beta > 0$, there exists a constant $\epsilon>0$ such that $\dfrac{\epsilon}{2}\leq \min(\tau,\beta)$,
\begin{equation*}\label{case4}
    \begin{aligned}
        B(\Delta_h,\tilde{D}_h) - S(\tilde{e}_u,\tilde{e}_v)
     =& \sum_j \left( \widehat{(\delta_u)_x}|_{j+1/2} [\![\tilde{e}_v]\!]_{j+1/2}+\widehat{\delta_v}|_{j+1/2} [\![(\tilde{e}_u)_x]\!]_{j+1/2}
     -\beta[\![\tilde{e}_v]\!]_{j+1/2}^2 
     -\tau [\![(\tilde{e}_u)_x]\!]_{j+1/2}^2 
     \right)\\
     \lesssim& \sum_j \dfrac{1}{2\epsilon}\left(\Vert\widehat{(\delta_u)_x}\Vert^2_{L^2(\partial I_j)} + \Vert\widehat{\delta_v}\Vert^2_{L^2(\partial I_j)}\right) + (\dfrac{\epsilon}{2}  - \tau)[\![(\tilde{e}_u)_x]\!]_{j+1/2}^2 + (\dfrac{\epsilon}{2} - \beta)[\![\tilde{e}_v]\!]^2_{j+1/2}\\
     \lesssim& \sum_j \dfrac{1}{\epsilon}\left(\Vert\widehat{(\delta_u)_x}\Vert^2_{L^2(\partial I_j)} + \Vert\widehat{\delta_v}\Vert^2_{L^2(\partial I_j)}\right)\\
     \lesssim&  h^{(2p-1)}|u|^2_{H^{p+1}(\Omega)} + h^{2q+1}|v|^2_{H^{q+1}(\Omega)}.
    \end{aligned}
\end{equation*}
In summary, we have the following result for above four cases,
\begin{equation}\label{case_all}
\begin{aligned}
    B(\Delta_h,\tilde{D}_h) - S(\tilde{e}_u,\tilde{e}_v)\lesssim\left(\Vert(\tilde{e}_u)_x)\Vert^2_{L^2(\Omega)} + \Vert\tilde{e}_v\Vert_{L^2(\Omega)}^2\right) + h^{2p^{\prime}}|u|^2_{H^{p+1}(\Omega)} + h^{2q^\prime}|v|_{H^{q+1}(\Omega)}^2 , 
\end{aligned}
\end{equation}
where $p', q'$ are given as 
    \begin{equation*}
    p' = 
    \begin{cases} 
    p-1, & \tau=0 , \\
    p-1/2, &\tau> 0,
    \end{cases} \qquad
    q' = 
    \begin{cases} 
    q, & \beta=0, \\
    q+1/2, &\beta> 0.
    \end{cases}
    \end{equation*}
For the damping term $D_u(u_h,\tilde{e}_u)$, according to the properties of $L^2$ projection, we can obtain,
\begin{equation}\label{eq:Du}
    \begin{aligned}
        D_u(u_h,\tilde{e}_u)=& \sum_j\sum_{l=1}^{p} \frac{\sigma_j^l}{h_j} \int_{I_j} \left((u_h)_x - \mathbb{P}^{l-1}(u_h)_x\right)(\tilde{e}_u)_x dx\\
        =&-\sum_j\sum_{l=1}^p\dfrac{\sigma_j^l}{h_j}\int_{I_j}((\tilde{e}_u)_x-\mathbb{P}^{l-1}(\tilde{e}_u)_x)(\tilde{e}_u)_x dx +\sum_j\sum_{l=1}^p\dfrac{\sigma_j^l}{h_j}\int_{I_j}((\tilde{u}_h)_x-\mathbb{P}^{l-1}(\tilde{u}_h)_x)(\tilde{e}_u)_x dx\\
        =&-\sum_j\sum_{l=1}^p\dfrac{\sigma_j^l}{h_j}\int_{I_j}((\tilde{e}_u)_x-\mathbb{P}^{l-1}(\tilde{e}_u)_x)^2 dx +\sum_j\sum_{l=1}^p\dfrac{\sigma_j^l}{h_j}\int_{I_j}((\tilde{u}_h)_x-\mathbb{P}^{l-1}(\tilde{u}_h)_x)(\tilde{e}_u)_x dx\\
        \leq&\sum_j\sum_{l=1}^p\dfrac{\sigma_j^l}{h_j}\int_{I_j}((\tilde{u}_h)_x-\mathbb{P}^{l-1}(\tilde{u}_h)_x)(\tilde{e}_u)_x dx\\
        \leq &\sum_j\sum_{l=1}^p\dfrac{\sigma_j^l}{h_j}\Vert(\tilde{u}_h)_x-\mathbb{P}^{l-1}(\tilde{u}_h)_x\Vert_{L^2(I_j)}\Vert(\tilde{e}_u)_x\Vert_{L^2(I_j)}.
    \end{aligned}
\end{equation}
By the properties of the $L^2$ projection, and since $u\in H^{p+1}(\Omega)$, we have
\begin{equation}\label{est:Pl}
    \begin{aligned}
        &\Vert(\tilde{u}_h)_x-\mathbb{P}^{l-1}(\tilde{u}_h)_x\Vert_{L^2(I_j)}\\
        &\leq\Vert(\tilde{u}_h-u)_x\Vert_{L^2(I_j)}+\Vert (u_x-\mathbb{P}^{l-1}u_x)\Vert_{L^2(I_j)}+\Vert(\mathbb{P}^{l-1}(u_x-(\tilde{u}_h)_x))\Vert_{L^2(I_j)}\\
        &\lesssim h_j^p|u|_{H^{p+1}(I_j)} + h_j^{\max{(1,l)}}|u_x|_{H^{\max(1,l)}(I_j)}+ h_j^p|u|_{H^{p+1}(I_j)} \\
        &\lesssim h_j^{\max{(1,l)}}|u|_{H^{\max(1,l)+1}(I_j)} \\
        &\lesssim h_j^{\max{(1,l)}}\left(\int_{I_j} |\partial^{\max{(1,l)+1}}u|^2dx\right)^{1/2}  \\
        &\lesssim h_j^{\max{(1,l)}+1/2}|\partial^{\max{(1,l)+1}}u|_{\infty}\\
        &\lesssim {h_j^{\max(1,l)+1/2}}, \qquad l=0,1,\ldots,p.
    \end{aligned}
\end{equation}
With the definition of $\sigma_j^l$, we have
\begin{equation}\label{est:sigma}
    \begin{aligned}
        (\sigma_j^l)^2=&\dfrac{4(2l+1)^2}{(2p-1)^2}\dfrac{h_j^{2l}}{(l!)^2}([\![\partial_x^l(u_h-u)]\!]_{j-1/2}^2+[\![\partial_x^l(u_h-u)]\!]_{j+1/2}^2)\\
        \lesssim & h_j^{2l}([\![\partial_x^l\tilde{e}_u]\!]_{j-1/2}^2+[\![\partial_x^l\tilde{e}_u]\!]_{j+1/2}^2)+h_j^{2l}([\![\partial_x^l\delta_u]\!]_{j-1/2}^2+[\![\partial_x^l\delta_u]\!]_{j+1/2}^2)
    \end{aligned}
\end{equation}
For the jump terms, we use the inverse inequality \eqref{eq:inverse} and the approximation property of projection \eqref{Ritz_Projection},
\begin{equation}\label{est:jump1}
    \begin{aligned}
        \sum_j\sum_{l=1}^ph_j^{2l}[\![\partial_x^l\tilde{e}_u]\!]_{j+1/2}^2\lesssim &
        \sum_j\sum_{l=1}^ph_j^{2l}(((\partial_x^l\tilde{e}_u)^+_{j+1/2})^2+((\partial_x^l\tilde{e}_u)_{j-1/2}^-)^2)\\
        \lesssim &h_j^{-1} \sum_j \sum_{l=1}^ph_j^{2l}\Vert\partial_x^l\tilde{e}_u\Vert^2_{L^2(I_j)}\\
        \lesssim &h_j\Vert(\tilde{e}_u)_x\Vert^2_{L^2(\Omega)}.
    \end{aligned}
\end{equation}
\begin{equation}\label{est:jump2}
    \begin{aligned}
        \sum_j\sum_{l=1}^ph_j^{2l}[\![\partial_x^l\delta_u]\!]_{j+1/2}^2\lesssim &
        \sum_j\sum_{l=1}^ph_j^{2l}(((\partial_x^l\delta_u)^+_{j+1/2})^2+((\partial_x^l\delta_u)_{j-1/2}^-)^2)\\
        \lesssim &h_j^{2p+1}|u|_{H^{p+1}(\Omega)}. 
    \end{aligned}
\end{equation}
Using estimates \eqref{est:Pl} - \eqref{est:jump2} in \eqref{eq:Du}, we obtain
\begin{equation}\label{Damping_u_error}
    \begin{aligned}
        D_u(u_h,\tilde{e}_u) 
        \lesssim&\dfrac{1}{h}\sum_j\sum_{l=1}^ph^{\max(l,1)+1/2+l}([\![\partial_x^l\tilde{e}_u]\!]_{j-1/2}^2+[\![\partial_x^l\tilde{e}_u]\!]_{j+1/2}^2)^{1/2}\Vert(\tilde{e}_u)_x\Vert_{L^2(I_j)}\\
        +&\dfrac{1}{h}\sum_j\sum_{l=1}^ph^{\max(l,1)+1/2+l}([\![\partial_x^l\delta_u]\!]_{j-1/2}^2+[\![\partial_x^l\delta_u]\!]_{j+1/2}^2)^{1/2}\Vert(\tilde{e}_u)_x\Vert_{L^2(I_j)}\\
        \lesssim& \left(\left(\sum_j\sum_{l=1}^ph^{2l}[\![\partial_x^l\tilde{e}_u]\!]_{j+1/2}^2\right)^{1/2} + \left(\sum_j\sum_{l=1}^ph^{2l}[\![\partial_x^l\delta_u]\!]_{j+1/2}^2\right)^{1/2}\right)h^{1/2}\Vert(\tilde{e}_u)_x\Vert_{L^2(\Omega)} \\
        \lesssim&h{\Vert(\tilde{e}_u)_x\Vert_{L^2(\Omega)}^2+h^{p+1}| u|_{H^{p+1}(\Omega)}\Vert(\tilde{e}_u)_x\Vert}_{L^2(\Omega)}\\
        \lesssim &{\Vert(\tilde{e}_u)_x\Vert_{L^2(\Omega)}^2+h^{2p+2}| u|^2_{H^{p+1}(\Omega)}}.
    \end{aligned}
\end{equation}
For the term $D_v(v_h,\tilde{e}_v)$,
in the same way, we have,
\begin{equation}\label{Damping_v_error}
    \begin{aligned}
        D_v(v_h,\tilde{e}_v)\leq &\sum_j\sum_{l=1}^q \dfrac{\tilde{\sigma}_j^l}{h_j} \Vert\tilde{v}_h-\mathbb{P}^{l-1}\tilde{v}_h\Vert_{L^2(I_j)} \Vert\tilde{e}_v\Vert_{L^2(I_j)}\\
        \lesssim&{\Vert\tilde{e}_v\Vert^2+h^{2q+2}|v|_{H^{q+1}(\Omega)}}.
    \end{aligned}
\end{equation}
Finally, we consider the term $P_u(u_h,\tilde{e}_u)$,
\begin{equation*}
    \begin{aligned}
        P_u(u_h,\tilde{e}_u) = &-\sum_j\dfrac{c}{h^2}([\![u_h]\!]_{j+1/2}\tilde{e}_u|_{j+1/2}^--[\![u_h]\!]_{j-1/2}\tilde{e}_u|_{j-1/2}^+)\\
        =&-\sum_j\dfrac{c}{h^2}([\![u_h-u]\!]_{j+1/2}\tilde{e}_u|_{j+1/2}^--[\![u_h-u]\!]_{j-1/2}\tilde{e}_u|_{j-1/2}^+)\\
        =&\sum_j\dfrac{c}{h^2}\left( [\![\tilde{e}_u-\delta_u]\!]_{j+1/2}\tilde{e}_u|_{j+1/2}^--[\![\tilde{e}_u-\delta_u]\!]_{j-1/2}\tilde{e}_u|_{j-1/2}^+ \right)\\
        =&\sum_j\dfrac{c}{h^2}\left( -[\![\tilde{e}_u]\!]_{j+1/2}^2 + [\![\delta_u]\!]_{j+1/2}[\![\tilde{e}_u]\!]_{j+1/2}\right).
    \end{aligned}
\end{equation*}
By Cauchy-Schwarz inequality, and the property of the Ritz projection \eqref{Ritz_Projection}, we can obtain, 
\begin{equation}\label{Penalty_u}
    \begin{aligned}
        P_u(u_h,\tilde{e}_u)\lesssim&\sum_j\dfrac{c}{h^2}\left(-[\![\tilde{e}_u]\!]_{j+1/2}^2 + \Vert\delta_u\Vert_{L^2(\partial I_j)}[\![\tilde{e}_u]\!]_{j+1/2}\right)\\
        \lesssim& \sum_j\dfrac{c}{h^2}\left(-[\![\tilde{e}_u]\!]_{j+1/2}^2 + \dfrac{1}{4} \Vert\delta_u\Vert_{L^2(\partial I_j)}^2 + [\![\tilde{e}_u]\!]_{j+1/2}^2\right)\\
        \lesssim& \sum_j \dfrac{c}{4h^2}\Vert\delta_u\Vert_{L^2(\partial I_j)}^2 \\
        \lesssim& h^{2p-1}|u|^2_{H^{p+1}(\Omega)}.
    \end{aligned}
\end{equation}
Therefore, coupling inequalities 
\eqref{case_all}, \eqref{Damping_u_error}-\eqref{Penalty_u}, we can obtain the following: 
\begin{equation*}
    \begin{aligned}
    \dfrac{1}{2} \dfrac{d}{dt} \sum_j\int_{I_j}\left( (\tilde{e}_u)_x^2 + (\tilde{e}_v)^2 \right) dx
    \lesssim &\sum_j\int_{I_j}\left( (\tilde{e}_u)_x^2 + (\tilde{e}_v)^2 \right) dx + h^{2\gamma} \left( |u|^2_{H^{p+1}(\Omega)} + |v|^2_{H^{q+1}(\Omega)} \right),
    \end{aligned}
\end{equation*}
where $\gamma=\min(p', q')$. Then by using Gronwall's inequality, we can obtain
\begin{equation*}
    \sum_j\int_{I_j}\left( (\tilde{e}_u)_x^2 + (\tilde{e}_v)^2 \right) dx
    \lesssim h^{2\gamma} \left( |u|^2_{H^{p+1}(\Omega)} + |v|^2_{H^{q+1}(\Omega)} \right).
\end{equation*}
The proof is completed by using the triangle inequality. 
\end{proof}

\begin{remark}
In the literature, the convergence rate  $p+1$ in $L^2$ norm for $u_h$ is typically considered optimal. Based on Theorem \ref{Thm_err}, we may expect a suboptimal convergence rate $p+1/2$ for S-flux and $p$ for A-flux and C-flux. However, our numerical examples show that both S-flux and A-flux achieve the optimal convergence rate. For C-flux, we have observed optimal convergence rate for odd $p$, and suboptimal convergence rate $p$ for even $p$. These observations are consistent with previously reported numerical results of EDG  \cite{appelo2015new, appelo2018energy, appelo2020energy}.
\end{remark}

\section{Multi-dimensional problems}\label{sec:multiD}
In this section, we extend the one-dimensional DG scheme to multi-dimensional problems. The stability analysis and a priori error estimate are also provided.

\subsection{The OF-EDG scheme in 2D}
We consider the model problem in a bounded domain $\mathbf{x}=(x,y)\in\Omega$ in $\mathbb{R}^2$,
\begin{equation}\label{eq:wave_2D}
    u_{tt} = \Delta u, \quad (\mathbf{x},t)\in \Omega\times(0,T ],
\end{equation}
with initial conditions
$$u(\mathbf{x},0) = u_0(\mathbf{x}),\qquad u_t(\mathbf{x},0) = u_1(\mathbf{x}).$$
We again consider periodic boundary conditions in each direction, and introduce the time derivative as a new variable $v = \dfrac{\partial u}{\partial t}$. The equation is rewritten as 
\begin{equation}
    \begin{cases}
        u_t = v,\\
        v_t = \Delta u,\\
    \end{cases} \quad (\mathbf{x},t)\in \Omega\times(0,T].
\end{equation}

For simplicity, we only present the formula for two-dimensional problems on a rectangular domain. We define   
$$K_{ij} = [x_{i-1/2},x_{i+1/2}]\times[y_{j-1/2},y_{j+1/2}],\qquad i = 1,\cdots,N_x, \, j = 1,\cdots,N_y,$$
with
$$h^x_i = x_{i+1/2} - x_{i-1/2},\quad 
h^y_j = y_{j+1/2} - y_{j-1/2}, $$
$$ h_{ij} = \sqrt{(h_i^x)^2 + (h_j^y)^2}, \quad h=\max h_{ij}.$$

Thus, the partition $\Gamma_h$ of $\Omega$ is 
\begin{equation}
\Gamma_h = \{K_{ij}\in\Omega,\, i=1, \ldots, N_x, j=1, \ldots, N_y\}.
\end{equation}
{We also define $\Gamma_e$ as the collection of the edges
\begin{equation}
\Gamma_e = \{e\in \partial K_{ij}, \,i=1, \ldots, N_x, j=1, \ldots, N_y \},
\end{equation}}
and define the discontinuous finite element space as 
\begin{equation}
V_h^k(\Omega) = \{w\in L^2(\Omega):w|_{K_{ij}}\in P^k(K_{ij}),\forall K_{ij}\in \Gamma_h \}.
\end{equation}

Following the idea in one dimension, we state the OF-EDG scheme in 2D as follows: find $u_h\in V^p_h(\Omega)$ and $v_h\in V^q_h(\Omega)$, such that for any $\phi^{(u)}(\mathbf{x})\in V^p_h(\Omega)$, $\phi^{(v)}(\mathbf{x})\in V^q_h(\Omega)$, $K_{ij}\in\Gamma_h$, 
\begin{equation}\label{eq:DG_2D}
    \begin{cases}
        \displaystyle
        \int_{K_{ij}}((u_h)_t-v_h)d\mathbf{x}&=0,\\
        \displaystyle
        \int_{K_{ij}}\nabla((u_h)_t-v_h)\cdot\nabla\phi^{(u)} d\mathbf{x}
        &=\displaystyle\int_{\partial K_{ij}} (\widehat{v} - v^-_h) (\nabla\phi^{(u),-} \cdot{\mathbf{n}})dS 
        + \dfrac{c}{h^2}\int_{\partial K_{ij}}([\![u_h]\!]\cdot\mathbf{n}) \phi^{(u),-} dS\\
         &\displaystyle\quad - \sum\limits_{l=1}^p\dfrac{\sigma_{ij}^l}{h_{ij}} \int_{K_{ij}} (\nabla u_h - \mathbb{P}^{l-1}\nabla u_h) \cdot\nabla\phi^{(u)} d\mathbf{x},   \\
        \displaystyle\int_{K_{ij}}((v_h)_t\phi^{(v)} + \nabla u_h\cdot\nabla\phi^{(v)})d\mathbf{x}
        &=\displaystyle\int_{\partial K_{ij}} ( \widehat{\nabla u}\cdot{\mathbf{n}}) \phi^{(v),-} dS 
        - \sum\limits_{l=0}^q\dfrac{\tilde{\sigma}_{ij}^l}{h_{ij}}\int_{K_{ij}}(v_h - \mathbb{P}^{l-1}v_h)\phi^{(v)} d\mathbf{x},
    \end{cases}
\end{equation}
where $\mathbf{n}$ represents the outward pointing normal of the element $K_{ij}$. 
The superscripts ``+" and ``-'' refer to traces of data from outside and inside the element, respectively. We define the jumps at the cell interface for a given function $w\in V^k_h$ as
\begin{equation}
    [\![w]\!] = w^+\mathbf{n}^+ + w^-\mathbf{n}^- \in \mathbb{R}^2, \qquad 
    {[\![\nabla w]\!] =\nabla w^+\cdot \mathbf{n}^+ + \nabla w^- \cdot \mathbf{n}^-} \in \mathbb{R} 
\end{equation}
{with $\bm{n}^-=\bm{n}$ and $\bm{n}^+=-\bm{n}$.}
As for the numerical fluxes $\widehat{\nabla u}$ and $\widehat{v}$, we again consider a general parameterization with $\alpha\in[0,1]$ and $\tau, \beta \geq0$. 
Using $\boldsymbol{\zeta} = (\alpha - \dfrac{1}{2},\alpha - \dfrac{1}{2})^T$, the numerical flux can be defined as follows,
\begin{equation}
    \begin{aligned}
     \widehat{\nabla u} &= \dfrac{\nabla u_h^+ + \nabla u_h^-}{2} - \left((\boldsymbol{\zeta}\cdot\mathbf{n}^+)\nabla u_h^+ + (\boldsymbol{\zeta}\cdot\mathbf{n}^-)\nabla u_h^-\right)-\beta[\![v_h]\!],\\
    \widehat{v} &= \dfrac{v_h^+ + v_h^-}{2} + \boldsymbol{\zeta}\cdot(v_h^+\mathbf{n}^+ + v_h^-\mathbf{n}^-) -\tau[\![\nabla u_h]\!]. \label{eq:flux_2D}
\end{aligned}
\end{equation}
It is easy to verify that with the following special parameters, the numerical fluxes are equivalent to those in 1D \eqref{eq:numerical_flux} along the horizontal or vertical cell-interface,
\begin{align}
\begin{array}{lll}
    & \text{Central flux (C-flux)}: & \alpha = \dfrac{1}{2}, \quad \beta = \tau = 0.\\
    & \text{Alternating flux (A-flux)}:& \alpha = 0 \,\, \text{or} \,\, 1, \quad \beta = \tau = 0.\\
    & \text{Sommerfeld flux (S-flux)}: & \alpha = \dfrac{1}{2}, \quad \beta = \dfrac{1}{2s}, \quad \tau = \dfrac{s}{2},\,\, s>0.
\end{array}
\end{align}

\begin{figure}[htbp]
    \centering
    \begin{tikzpicture}[scale = 1.2]
        \draw[dashed] (0, 0) grid (4, 4);
        \node at (1.5, 1.5) {$K_{i-1,j}$};
        \node at (2.5, 2.5) {$K_{i,j+1}$};
        \node at (1.7, 2.3) {$\mathbf{v}$};
        \node at (2.5, 1.5) {$K_{ij}$};
        \fill[red] (2, 2) circle (3pt);
    \end{tikzpicture}
    \caption{Illustration of the jump term in (\ref{jump_v}), where the red solid point represents a vertex $\mathbf{v}$ of $K_{ij}$, and the jumps on the faces $K_{ij}\cap K_{i-1,j}$ and $K_{ij}\cap K_{i,j+1}$ are used in the definition.}
    \label{fig:2D}
\end{figure}
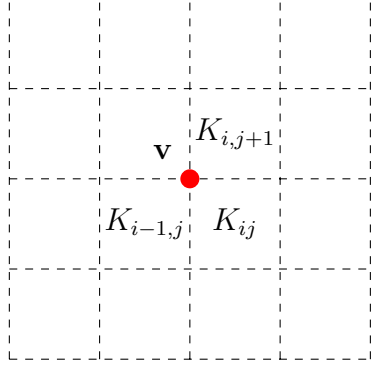

\noindent
Following the idea in \cite{lu2021oscillation}, the damping coefficients $\sigma_{ij}^l$ and $\tilde{\sigma}_{ij}^l$ depend on the jumps at the vertices of $K_{ij}$. Let $[\![w]\!]|_{\mathbf{v}}$ denote the jump of $w$ on element $K_{ij}$ and its adjacent elements at vertex $\mathbf{v}$. 
As illustrated in Figure \ref{fig:2D}, at vertex $\mathbf{v}=(x_{i-1/2}, y_{j+1/2})$ in $K_{ij}$, the term $[\![w]\!]|_{\mathbf{v}}$  takes into account only jumps across faces that contain this vertex. More precisely, we define
\begin{equation}\label{jump_v}
    [\![w]\!]|_{\mathbf{v}}^2 = \left( w|_{K_{ij}}(\mathbf{v})-w|_{K_{i-1,j}}(\mathbf{v}) \right)^2 + \left( w|_{K_{ij}}(\mathbf{v})-w|_{K_{i,j+1}}(\mathbf{v}) \right)^2.
\end{equation}
Additionally, the coefficients of the damping term $\sigma_{ij}^l \geq 0, \tilde{\sigma}_{ij}^l \geq 0$ are taken as follows,
\begin{equation}\label{jump}
    \begin{aligned}
        \sigma_{ij}^l &= \dfrac{2(2l+1)}{(2p-1)}\dfrac{h_{ij}^l}{l!} \sum\limits_{|\boldsymbol{\alpha}|=l} \left( \dfrac{1}{4} \sum\limits_{\mathbf{v}\in K_{ij}} ([\![\partial^{\boldsymbol{\alpha}}u_h]\!]|_\mathbf{v})^2 \right)^{1/2},
        \quad l\geq 1, \\
        \tilde{\sigma}_{ij}^l & = \dfrac{2(2l+1)}{(2q-1)}\dfrac{h_{ij}^{l+1}}{(l+1)!} \sum\limits_{|\boldsymbol{\alpha}|=l} \left( \dfrac{1}{4} \sum\limits_{\mathbf{v}\in K_{ij}} ([\![\partial^{\boldsymbol{\alpha}}v_h]\!]|_\mathbf{v})^2 \right)^{1/2},
        \quad l\geq 0.
    \end{aligned}
\end{equation}
Here, the vector $\boldsymbol{\alpha} = (\alpha_1,\alpha_2)$
is the multi-index order with
$|\boldsymbol{\alpha}| = \alpha_1+\alpha_2,$
and $\partial^{\boldsymbol{\alpha}}w$ is defined as
\begin{equation*}
\partial^{\boldsymbol{\alpha}}w = \dfrac{\partial^{|\boldsymbol{\alpha}|}w}{\partial x^{\alpha_1} \partial y^{\alpha_2}}.
\end{equation*}
The operator $\mathbb{P}^l$ is the standard $L^2$ projection onto $V^l_h$, that for any function $w$ construct $\mathbb{P}^{l}w \in V^l_h$ such that
\begin{equation}
    \int_{K_{ij}}(\mathbb{P}^{l}w - w)\, \phi d\mathbf{x} = 0, 
    \quad \forall \, \phi(\mathbf{x})\in V^l_h.
\end{equation}
And we use the convention $\mathbb{P}^{-1} = \mathbb{P}^0$ as well.

\subsection{Stability analysis}
In this subsection, we establish a stability proof for the multi-dimension problems with general numerical fluxes. 

\begin{theorem}[Semi-discrete stability]
    Under the assumption of periodic boundary condition in each direction, for any positive integers $p$ and $q$, the semi-discrete DG scheme \eqref{eq:DG_2D} with the numerical fluxes \eqref{eq:flux_2D} satisfies 
    \begin{equation}
    \dfrac{d\mathbb{E}_h}{dt}\leq 0,
    \end{equation} 
    as long as the parameters $c$, $\tau$, and $\beta$ are not less than 0,
    and the semi-discrete energy $\mathbb{E}_h$ is defined as,
    \begin{equation}
        \mathbb{E}_h = \int_{\Omega} \left( \nabla u_h\cdot\nabla u_h + v_h^2 \right) d\mathbf{x}.
    \end{equation}
\end{theorem}

\begin{proof}
    Taking test function $(\phi^{(u)}(\mathbf{x}),\phi^{(v)}(\mathbf{x})) = (u_h,v_h)$ and summing over all $K_{ij}$, we  can obtain
    \begin{equation*}
        \begin{aligned}
             \dfrac{1}{2}\dfrac{d}{dt}& \int_{\Omega} \left( \nabla u_h\cdot\nabla u_h + v_h^2 \right) d\mathbf{x} \\
             =&
    \sum\limits_{K_{ij}\in\Gamma_h} \int_{K_{ij}} \left( \nabla(u_h)_t\cdot\nabla u_h + (v_h)_tv_h \right)d\mathbf{x}\\
     =& \sum\limits_{K_{ij}\in\Gamma_h} \int_{\partial K_{ij}}(\widehat{v} - v_h^-)(\nabla u_h^{-}\cdot\mathbf{n})dS +\sum\limits_{K_{ij}\in\Gamma_h} \int_{\partial K_{ij}}(\widehat{\nabla u}\cdot\mathbf{n})\, v_h^- dS\\
    &+\sum\limits_{K_{ij}\in\Gamma_h} \dfrac{c}{h^2} \int_{\partial K_{ij}}([\![u_h]\!]\cdot\mathbf{n}) u_h^- dS 
    -\sum\limits_{K_{ij}\in\Gamma_h}\sum\limits_{l=1}^p\dfrac{\sigma_{ij}^l}{h_{ij}} \int_{K_{ij}}(\nabla u_h - \mathbb{P}^{l-1}\nabla u_h)\cdot\nabla u_h d\mathbf{x} \\
    &-\sum\limits_{K_{ij}\in\Gamma_h}\sum\limits_{l=0}^q\dfrac{\tilde{\sigma}_{ij}^l}{h_{ij}} \int_{K_{ij}}(v_h - \mathbb{P}^{l-1}v_h)v_h d\mathbf{x}. 
        \end{aligned}
    \end{equation*}
    Plugging in the numerical fluxes and using the periodic boundary condition yields
    \begin{align*}
        \dfrac{1}{2}\dfrac{d}{dt}& \int_{\Omega} (\nabla u_h\cdot\nabla u_h + v_h^2) d\mathbf{x} \\
        = &    -\tau\sum_{e\in \Gamma_e} \int_{e}[\![\nabla u_h]\!]^2 dS
        -\beta\sum_{e}\int_{e\in \Gamma_e}\Vert[\![v_h]\!]\Vert^2 dS
        - \dfrac{c}{h^2}\sum_{e\in \Gamma_e} \int_e \Vert[\![u_h]\!]\Vert^2 dS\\ 
    &-\sum\limits_{K_{ij}\in\Gamma_h} \sum\limits_{l=0}^q \dfrac{\tilde{\sigma}_{ij}^l}{h_{ij}}\int_{K}(v_h - \mathbb{P}^{l-1}v_h)^2d\mathbf{x}\\
    &- \sum\limits_{K_{ij}\in\Gamma_h}\sum\limits_{l=1}^p 
 \dfrac{\sigma_{ij}^l}{h_{ij}} \int_{K_{ij}}(\nabla u_h - \mathbb{P}^{l-1}\nabla u_h)^T(\nabla u_h - \mathbb{P}^{l-1}\nabla u_h)d\mathbf{x},
        \end{align*}
        where $\Vert\cdot\Vert$ represents the $L_2$ norm.
    Utilizing the fact that these parameters $c\geq 0$, $\tau\geq 0$, $\beta \geq 0$, we obtain the desired energy estimate.
\end{proof}

\subsection{A prior error estimate }
\begin{theorem}[Error estimate]
    Assume the exact solutions $u\in H^{p+1}(\Omega)$ and $v\in H^{q+1}(\Omega)$
    with periodic boundary condition in each direction.     If $p-2\leq q\leq p$, then  
    \begin{equation}\label{eq:error_2D}
        \Vert\nabla(u-u_h)\Vert^2_{L^2(\Omega)} + \Vert v-v_h\Vert_{L^2(\Omega)}^2 \lesssim h^{2\gamma}\left( |u|^2_{H^{p+1}(\Omega)} +|v|^2_{H^{q+1}(\Omega)} \right),
    \end{equation}
    where $\gamma=\min(p', q')$, and
    \begin{equation}
    p' = 
    \begin{cases} 
    p-1, & \tau=0 , \\
    p-1/2, &\tau> 0,
    \end{cases} \qquad
    q' = 
    \begin{cases} 
    q, & \beta=0, \\
    q+1/2, &\beta> 0.
    \end{cases}
    \end{equation}
\end{theorem}

\begin{proof}
    We denote $V^{p,q}=V_h^p\times V_h^q$ to be the set of all functions $U_h = (u_h,v_h)$,  
$U = (u,v)$ is the exact solution, and $\Phi = (\phi^{(u)},\phi^{(v)})\in V^{p,q}$ is the test function. Similar to one dimensional case, we define
\begin{equation*}
\begin{aligned}
    D_{u}(u_h, \phi^{(u)}) = \sum_{K_{ij}\in\Gamma_h} D_{u,K_{ij}}(u_h, \phi^{(u)}), \quad 
    D_{u,K_{ij}}(u_h, \phi^{(u)}) = & \sum_{l=1}^{p} \frac{\sigma_{{ij}}^l}{h} \int_{K_{ij}} \left(\nabla u_h - \mathbb{P}^{l-1}\nabla u_h\right)\cdot\nabla\phi^{(u)} d\mathbf{x},\\
     D_{v}(v_h, \phi^{(v)}) = \sum_{K_{ij}\in\Gamma_h} D_{v,K_{ij}}(v_h, \phi^{(v)}), \quad 
     D_{v,K_{ij}}(v_h, \phi^{(v)}) = & \sum_{l=0}^{q} \frac{\tilde{\sigma}_{{ij}}^l}{h} \int_{K_{ij}} \left(v_h - \mathbb{P}^{l-1}v_h \right) \phi^{(v)} d\mathbf{x},\\
     P_{u}(u_h,\phi^{(u)}) = \sum_{K_{ij}\in\Gamma_h} P_{u,K_{ij}}(u_h,\phi^{(u)}), \quad 
     P_{u,K_{ij}}(u_h,\phi^{(u)})=&- \dfrac{c}{h^2}\int_{\partial K_{ij}}([\![u_h]\!]\cdot\mathbf{n}) \phi^{(v),-}(\mathbf{x})dS.
\end{aligned}
\end{equation*}
We also define the error between exact and numerical solution
\begin{align*}
    e_u = u-u_h,\quad e_v=v-v_h,\quad D_h =(e_u,e_v),
\end{align*}
and split the error into $e_u=\tilde{e}_u-\delta_u$  and $e_v=\tilde{e}_v-\delta_v$ with  
\begin{equation*}
\begin{aligned}
    & \tilde{e}_u=\tilde{u}_h-u_h,\quad \tilde{e}_v=\tilde{v}_h-v_h,\quad \tilde{D}_h=(\tilde{e}_u,\tilde{e}_v)\in V^{p,q},\\
    & \delta_u=\tilde{u}_h-u,\quad\,\,\, \delta_v = \tilde{v}_h-v,\quad\,\,\, \Delta_h = (\delta_u,\delta_v),
\end{aligned}
\end{equation*}
where $\tilde{v}_h=\mathbb{P}^q v \in V_h^q$ is the $L^2$ projection of $v$, 
and $\tilde{u}_h\in V^p_h$ is obtained via a Ritz-type projection in multi-dimensions 
\begin{equation*}
\begin{cases}
    \displaystyle\int_{K_{ij}} \nabla(u-\tilde{u}_h)\cdot\nabla\phi^{(u)}d\mathbf{x} = 0 , \quad \forall \,\phi^{(u)}\in V^{p}_h,\\
    \displaystyle\int_{K_{ij}} (u - \tilde{u}_h)d\mathbf{x} = 0.
\end{cases}
\end{equation*}
The properties \eqref{eq:prop_add1}-\eqref{eq:inverse} also hold in 2D.

Moreover, choosing $\Phi = \tilde{D}_h$ and summing up over $K\in\Gamma_h$, we  obtain 
\begin{equation*}\label{Eh_2D}
\begin{aligned}
    \dfrac{1}{2} \dfrac{d}{dt} & \sum_j\int_{K_{ij}} \left( \nabla\tilde{e}_u\cdot\nabla\tilde{e}_u + \tilde{e}_v^2 \right) d\mathbf{x}\\
    =& -\sum_{e\in\Gamma_e}\int_{e} \left( [\![\nabla\tilde{e}_u]\!] \, \widehat{\delta_v} + [\![\tilde{e}_v]\!]\cdot\widehat{\nabla\delta_u} +  \beta [\![\tilde{e}_v]\!]^2+\tau [\![\nabla\tilde{e}_u]\!]^2 
     \right) dS\\
    &+D_u(u_h,\tilde{e}_u) +D_v(v_h,\tilde{e}_v) +P_u(u_h,\tilde{e}_u).
\end{aligned}
\end{equation*}

Similar to the 1D error estimate, the first term on the right-hand side  can also be estimated using the properties \eqref{eq:prop_add1}-\eqref{eq:inverse}. Consequently, we have

\begin{equation}\label{eq:case_2D}
    \begin{aligned}
        -\sum_{e\in\Gamma_e}\int_{e} & \left( [\![\nabla\tilde{e}_u]\!] \, \widehat{\delta_v} + [\![\tilde{e}_v]\!]\cdot\widehat{\nabla\delta_u} +  \beta [\![\tilde{e}_v]\!]^2+\tau [\![\nabla\tilde{e}_u]\!]^2 
     \right) dS \\
     & \lesssim \left(\Vert\nabla\tilde{e}_u\Vert^2_{L^2(\Omega)} + \Vert\tilde{e}_v\Vert^2_{L^2(\Omega)}\right) + h^{2p^\prime}|u|^2_{H^{p+1}(\Omega)} + h^{2q^\prime}|v|^2_{H^{q+1}(\Omega)},
    \end{aligned}
\end{equation}
where $p', q'$ are given as
    \begin{equation*}
    p' = 
    \begin{cases} 
    p-1, & \tau=0 , \\
    p-1/2, &\tau> 0,
    \end{cases} \qquad
    q' = 
    \begin{cases} 
    q, & \beta=0, \\
    q+1/2, &\beta> 0.
    \end{cases}
    \end{equation*}

With the definition of $\sigma_{ij}^l$, by inverse inequality \eqref{eq:inverse} and the approximation properties of projection, we have
\begin{equation*}
    \begin{aligned}
        \sum_{K_{ij}\in\Gamma_h}(\sigma_{ij}^l)^2 =& \sum_{K_{ij}\in\Gamma_h}\dfrac{4(2l+1)^2}{(2p-1)^2} \dfrac{h_{ij}^{2l}}{(l!)^2} \sum_{|\boldsymbol{\alpha}| = l}\left(\dfrac{1}{4}\sum\limits_{\mathbf{v}\in K_{ij}}([\![\partial^{\boldsymbol{\alpha}}u_h-\partial^{\boldsymbol{\alpha}}u]\!]|_{v})^2 \right)\\
        \lesssim& \sum_{K_{ij}\in\Gamma_h} \sum_{|\boldsymbol{\alpha}|=l} h_{ij}^{2l} \sum\limits_{\mathbf{v}\in K_{ij}} \left([\![\partial^{\boldsymbol{\alpha}}\tilde{e}_u]\!]^2|_{\mathbf{v}} + [\![\partial^{\boldsymbol{\alpha}}\delta_u]\!]^2|_{\mathbf{v}}\right)\\
         \lesssim& \Vert\nabla\tilde{e}_u\Vert_{L^2(\Omega)}^2 + h^{2p}|u|^2_{H^{p+1}(\Omega)}.
    \end{aligned}
\end{equation*}
On the other hand, 
\begin{equation*}
    \begin{aligned}
        &\Vert\nabla\tilde{u}_h-\mathbb{P}^{l-1}\nabla\tilde{u}_h\Vert_{L^2(K_{ij})}\\
        &\leq\Vert\nabla(\tilde{u}_h-u)\Vert_{L^2(K_{ij})}+\Vert (\nabla u-\mathbb{P}^{l-1}\nabla u)\Vert_{L^2(K_{ij})}+\Vert\mathbb{P}^{l-1}\nabla(u-\tilde{u}_h)\Vert_{L^2(K_{ij})}\\
        &\lesssim h_{ij}^p|u|_{H^{p+1}(K_{ij})} + h_{ij}^{\max(1,l)}|\nabla u|_{H^{\max(1,l)}(K_{ij})} + h_{ij}^p|u|_{H^{p+1}(K_{ij})}\\
        &\lesssim h_{ij}^{\max(1,l)}|u|_{H^{\max(1,l)+1}(K_{ij})}\\
        &\lesssim h_{ij}^{\max(1,l)}\left(\sum_{|\boldsymbol{\alpha}| = \max(1,l)+ 1} \int_{K_{ij}} |\partial^{\boldsymbol{\alpha}}u|^2dxdy\right)^{1/2}\\
        &\lesssim h_{ij}^{\max(1,l)+1}\sum_{|\boldsymbol{\alpha}| = \max(1,l)+ 1}|\partial^{\boldsymbol{\alpha}}u|_{\infty}\\
        &{\lesssim h^{\max(1,l)+1}}, \qquad l=0,1,\ldots,p.
    \end{aligned}
\end{equation*}
Hence, we obtain
\begin{equation}\label{Damping_u_error_2D}
    \begin{aligned}
        D_u(u_h,\tilde{e}_u)
        =&-\sum_{K_{ij}\in\Gamma_h}\sum_{l=1}^p\dfrac{\sigma_{ij}^l}{h_{ij}}\int_{K_{ij}}\left((\nabla\tilde{e}_u-\mathbb{P}^{l-1}\nabla\tilde{e}_u)^2 + (\nabla\tilde{u}_h-\mathbb{P}^{l-1}\nabla\tilde{u}_h)\nabla\tilde{e}_u \right)d\Omega\\
        \leq &\sum_{K_{ij}\in\Gamma_h}\sum_{l=1}^p\dfrac{\sigma_{ij}^l}{h_{ij}}\Vert\nabla\tilde{u}_h-\mathbb{P}^{l-1}\nabla\tilde{u}_h\Vert_{L^2(K_{ij})}\Vert\nabla\tilde{e}_u\Vert_{L^2(K_{ij})}\\
        \lesssim& \sum_{K_{ij}\in\Gamma_h}\sum_{l=1}^p\dfrac{(\sigma_{ij}^l)^2}{h_{ij}^2}\Vert\nabla\tilde{u}_h-\mathbb{P}^{l-1}\nabla\tilde{u}_h\Vert_{L^2(K_{ij})}^2 +\Vert\nabla\tilde{e}_u\Vert^2_{L^2(\Omega)}\\
        \lesssim& \sum_{K_{ij}\in\Gamma_h}\sum_{l=1}^p\dfrac{(\sigma_{{ij}}^l)^2}{h_{ij}^2}h^{2\max(1,l) + 2} + \Vert\nabla\tilde{e}_u\Vert^2_{L^2(\Omega)}\\
        \lesssim& \Vert\nabla\tilde{e}_u\Vert^2_{L^2(\Omega)} + h^{2p+2}|u|^2_{H^{p+1}(\Omega)}.
    \end{aligned}
\end{equation}
For the term $D_v(v_h,\tilde{e}_v)$,
 we have,
\begin{equation}\label{Damping_v_error_2D}
    \begin{aligned}
        D_v(v_h,\tilde{e}_v)\leq &\sum_{K_{ij}\in\Gamma_h}\sum_{l=1}^q \dfrac{\tilde{\sigma}_{{ij}}^l}{h_{ij}} \Vert\tilde{v}_h-\mathbb{P}^{l-1}\tilde{v}_h\Vert_{L^2(K_{ij})} \Vert\tilde{e}_v\Vert_{L^2(K_{ij})}\\
        \lesssim&{\Vert\tilde{e}_v\Vert^2+h^{2q+2}|v|^2_{H^{q+1}(\Omega)}}.
    \end{aligned}
\end{equation}
Finally, we consider the term $P_u(u_h,\tilde{e}_u)$,
\begin{equation}\label{Penalty_u_2D}
    \begin{aligned}
        P_{u}(u_h,\tilde{e}_u)
        =&-\sum_{K_{ij}\in\Gamma_h} \dfrac{c}{h^2} \int_{\partial K_{ij}} ([\![u_h]\!] \cdot\mathbf{n}^-)\tilde{e}^-_udS\\
        = &-\dfrac{c}{h^2} \sum_{e\in\Gamma_e}\int_{e} \left([\![u_h - u]\!] \cdot [\![\tilde{e}_u]\!]\right)dS\\
        =& \dfrac{c}{h^2} \sum_{e\in\Gamma_e}\int_{e} \left(-\Vert[\![\tilde{e}_u ]\!]\Vert^2 + [\![\delta_u]\!] \cdot [\![\tilde{e}_u]\!]\right) dS \\
        \lesssim& \dfrac{c}{h^2} \sum_{e\in\Gamma_e} \int_{e} \left(-\Vert[\![\tilde{e}_u ]\!]\Vert^2 + \dfrac{1}{4}\Vert[\![\delta_u]\!]\Vert^2 +  \Vert[\![\tilde{e}_u ]\!]\Vert^2  \right) dS \\
        \lesssim& \dfrac{c}{h^2}\sum_{e\in\Gamma_e}\int_{e} \Vert[\![\delta_u]\!]\Vert^2  dS \\
        \lesssim& h^{2p-1}|u|^2_{H^{p+1}(\Omega)}.
    \end{aligned}
\end{equation}
Therefore, by combining inequalities \eqref{eq:case_2D}-\eqref{Penalty_u_2D}, we  obtain 
\begin{equation*}
    \begin{aligned}
    \dfrac{d}{dt} \int_{\Omega}\left( \nabla\tilde{e}_u\cdot\nabla\tilde{e}_u + (\tilde{e}_v)^2 \right) d\mathbf{x}
    \lesssim &\int_{\Omega}\left( \nabla\tilde{e}_u\cdot\nabla\tilde{e}_u + (\tilde{e}_v)^2 \right) d\mathbf{x} + h^{2\gamma} \left( |u|^2_{H^{p+1}(\Omega)} + |v|^2_{H^{q+1}(\Omega)} \right),
    \end{aligned}
\end{equation*}
where $\gamma=\min(p', q')$. Then, using Gronwall's inequality yields 
\begin{equation*}
    \int_{\Omega}\left( \nabla\tilde{e}_u\cdot\nabla\tilde{e}_u + (\tilde{e}_v)^2 \right) d\mathbf{x}
    \lesssim h^{2\gamma} \left( |u|^2_{H^{p+1}(\Omega)} + |v|^2_{H^{q+1}(\Omega)} \right).
\end{equation*}
We complete the proof by using the triangle inequality.
\end{proof}

\section{Numerical results} \label{sec:num}

In this section, we present numerical tests to validate our theoretical results in one and two dimensions.
We discretize in time using the third order strong stability-preserving Runge-Kutta (SSP-RK3) method \cite{shu1988efficient}
for all cases. In particular, the time step is chosen as $dt = h/20, h^{4/3}/20, h^{5/3}/20, h^2/20, h^{7/3}/20$ for $p = 2,3,4,5,6$ respectively to match the accuray order in the spatial discretization. We use a default final time $t = 0.25$ for the following smooth numerical examples, reporting the $L^2$ error in displacement $u_h$ with A-Flux, S-Flux and C-Flux for smooth problems.
For problems with discontinuities, to avoid repetition, we only show the results with A-Flux and plot the numerical solutions and the reference solutions, where we use red solid lines to represent exact solutions and blue circles to represent numerical solutions at the midpoint of each element. For the approximation of finite element spaces, we  choose $q=p-1$.  
The value of the penalty parameter is taken as $c = 1$ in the experiments.

\begin{remark}
    We can extend our scheme \eqref{eq:DG_1D} to solve the wave equation with a nonlinear source term, 
\begin{equation}\label{nonlinear_1D}
        u_{tt} = u_{xx} + g(u).
\end{equation}
Based on \cite{appelo2020energy}, the corresponding semi-discrete DG scheme is 
\begin{equation}\label{eq:DGnonlinear_1D}
\begin{cases}
    \displaystyle
    {\int_{I_j} \left( (u_h)_t - v_h \right) dx }&=0, \\
    \displaystyle
    \int_{I_j} \left( (u_h)_t - v_h \right)_x \phi^{(u)}_{x} dx 
    &= \left(\widehat{v}_{j+1/2} - v_h|^{-}_{j+1/2} \right) \phi_x^{(u)}|^{-}_{j+1/2} -
    \left( \widehat{v}_{j-1/2} - v_h|^{+}_{j-1/2} \right) \phi_x^{(u)}|^{+}_{j-1/2} \\
    &\quad +\chi\displaystyle\int_{I_j}\phi^{(u)}\dfrac{g(u_h)}{u_h}\left(\dfrac{\partial u_h}{\partial t} - v_h\right)dx\\
    &\quad
    +\dfrac{c}{h^2} \left( [\![u_h]\!]_{j+1/2} \phi^{(u)}|^-_{j+1/2} -[\![u_h]\!]_{j-1/2} \phi^{(u)}|^+_{j-1/2} \right)\\
    &\quad{\displaystyle  - \sum_{l=1}^{p} \frac{\sigma_j^l}{h_j} \int_{I_j} \left((u_h)_x - \mathbb{P}^{l-1}(u_h)_x\right) \phi^{(u)}_x dx},\qquad \forall {\phi^{(u)}\in V_h^p(I_j)}, \\
    \displaystyle \int_{I_j} \left( (v_h)_t \phi^{(v)} + (u_h)_x \phi^{(v)}_{x} \right) dx 
    &= \widehat{u_x}|_{j+1/2}\phi^{(v)}|^{-}_{j+1/2} - 
    \widehat{u_x}|_{j-1/2} \phi^{(v)}|^{+}_{j-1/2}{\displaystyle +\int_{I_j}g(u_h)\phi^{(v)}dx}\\
    &\quad  - \sum\limits_{l=0}^{q} \dfrac{\tilde{\sigma}_j^l}{h_j} \displaystyle\int_{I_j} \left(v_h - \mathbb{P}^{l-1}v_h \right) \phi^{(v)} dx,
    \qquad \forall {\phi^{(v)}\in V_h^q(I_j)}.
\end{cases}
\end{equation}
When $\chi = 1$, the nonlinear term is treated in the same way as in \cite{appelo2020energy}. In this case, assuming $\lim\limits_{u\to 0}g(u)/u$ is bounded and  $G(u) := -\displaystyle\int_0^ug(z)dz > 0$, the scheme \label{nonlinear_1D} satisfies an energy estimate with the discrete energy
\begin{equation}
    \mathbb{E}_h(t) = \sum\limits_j\left(\dfrac{1}{2}\int_{I_j} \left(v_h^2 + (u_h)_x^2\right) dx + \int_{I_j}G(u_h(x))dx\right). 
\end{equation}
We have carried out several numerical tests of solving the wave equation with a nonlinear source term inspired by \cite{Cao2025}.  
In our numerical tests, we have also used $\chi = 0$ and observed stable results with significant speed up. Therefore, we use the scheme with $\chi = 1$ for one-dimensional problems, and $\chi = 0$ for two-dimensional problems in the numerical tests. 
\end{remark}

\subsection{One-dimensional problems}

\begin{example}\label{ex1} (Accuracy test for linear problem)\end{example}

We consider equation \eqref{eq:wave_1D} in domain $\Omega = (-1,1)$ with exact solution $u(x,t) = \sin(\pi(x-t))$, which is also used to obtain the initial data. 
In Figure \ref{fig:ex1}, we show the $L^2$ error of $u_h$ on uniform meshes at the final time for different choices of numerical fluxes and polynomial degrees. We can observe the optimal convergence rate $p+1$ with either S-flux or A-flux for all cases. For C-flux, the optimal convergence rate is obtained when the polynomial degree $p$ is odd, whereas suboptimal convergence rate reduced by one is obtained with even $p$. This result is consistent with the original EDG method in \cite{appelo2015new}.  

\begin{figure}[!ht] 
\centering
\subfigure[A-flux.]{\includegraphics[width=3.0in]{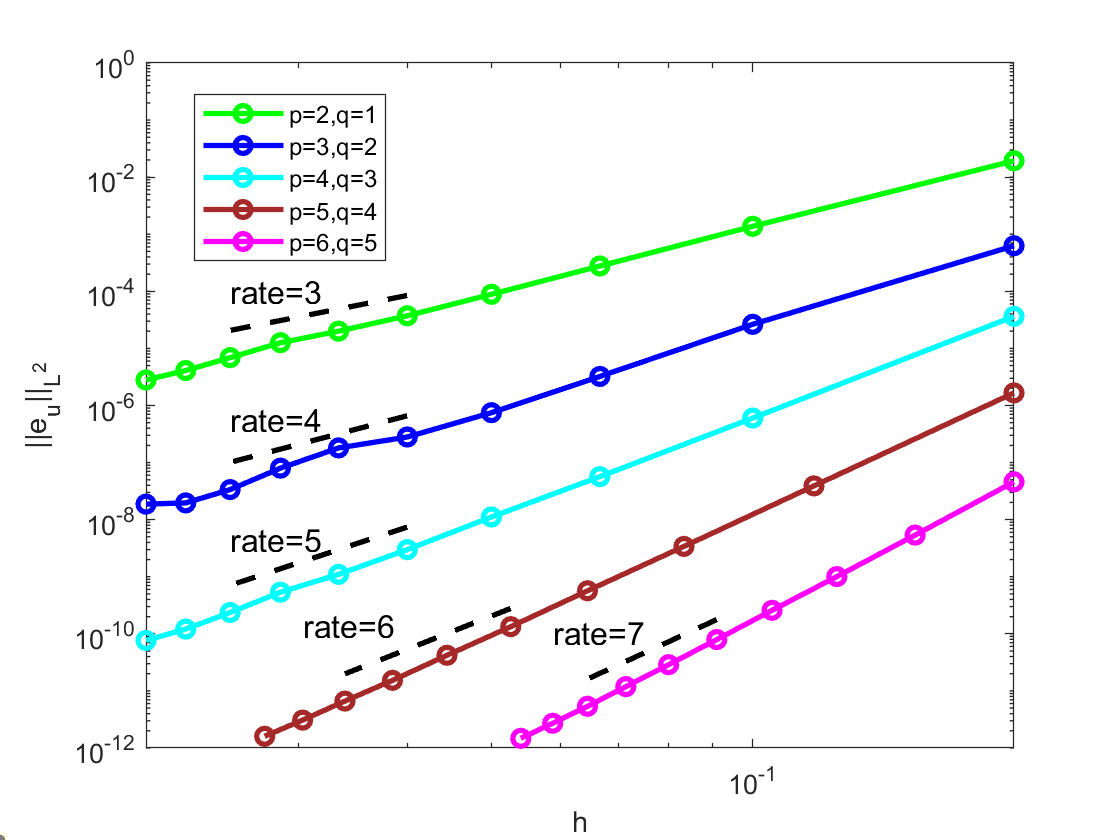}}
\subfigure[S-flux.]{\includegraphics[width=3.0in]{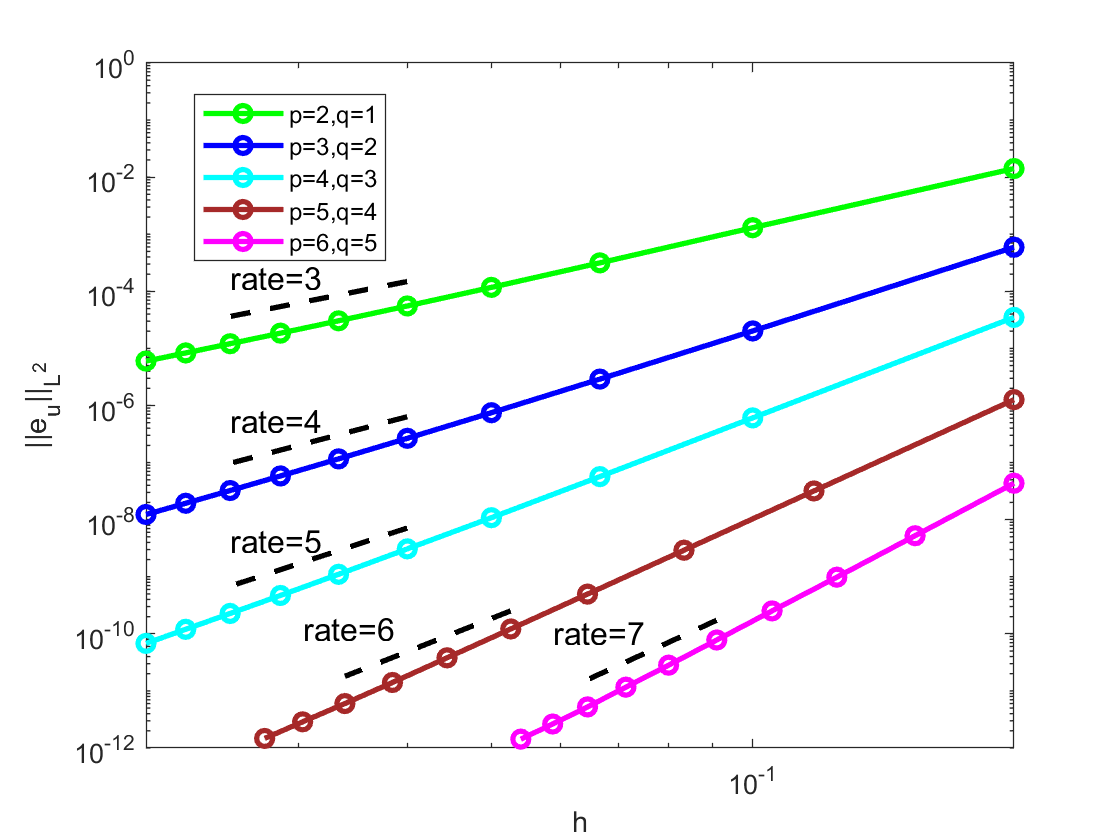}}
\subfigure[C-flux.]{\includegraphics[width=3.0in]{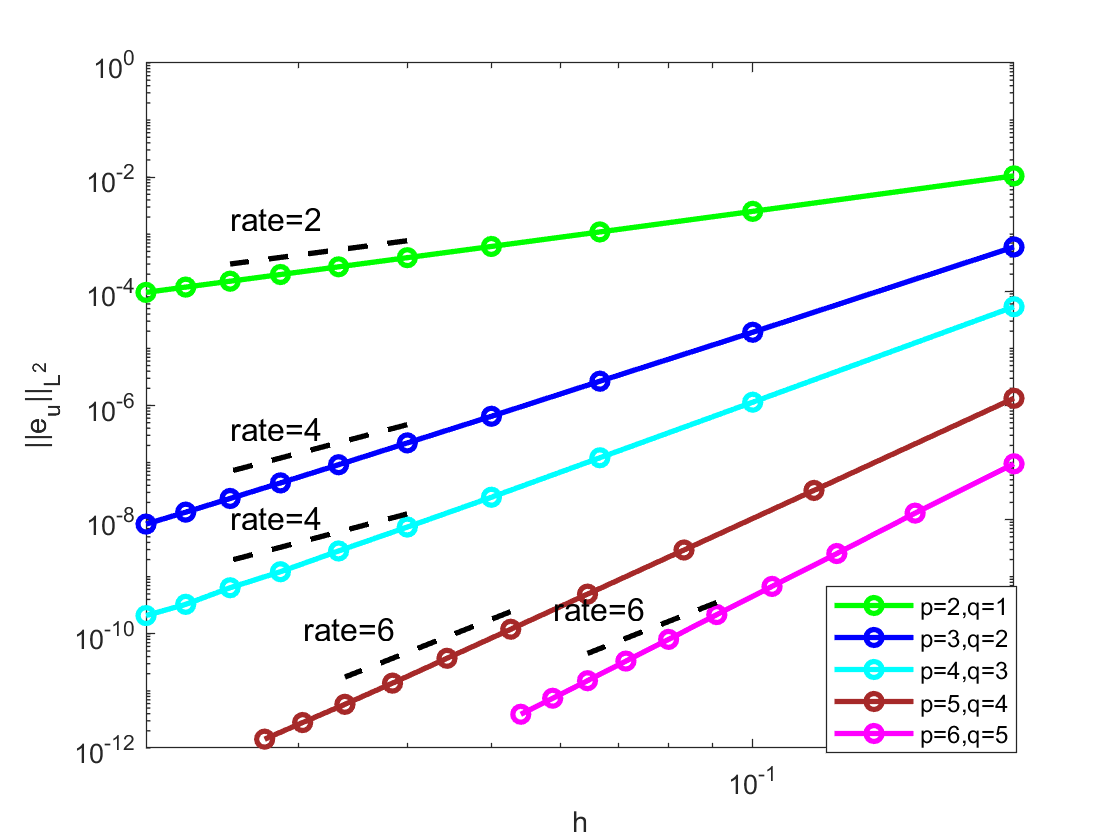}}
\caption{Example \ref{ex1}: $L^2$ errors with different numerical fluxes on uniform meshes.}
\label{fig:ex1}
\end{figure}

We have also performed the same convergence test on nonuniform meshes, created by randomly perturbing all internal nodes on a uniform mesh by up to 10\% of its mesh size. In Figure \ref{fig:ex1_nonuniform}, we show the $L^2$ error of $u_h$ at the final time for different choices of numerical fluxes and polynomial degrees, and we observe the same convergence rates as the cases with uniform meshes.

\begin{figure}[!ht] 
\centering
\subfigure[A-flux.]{\includegraphics[width=3.0in]{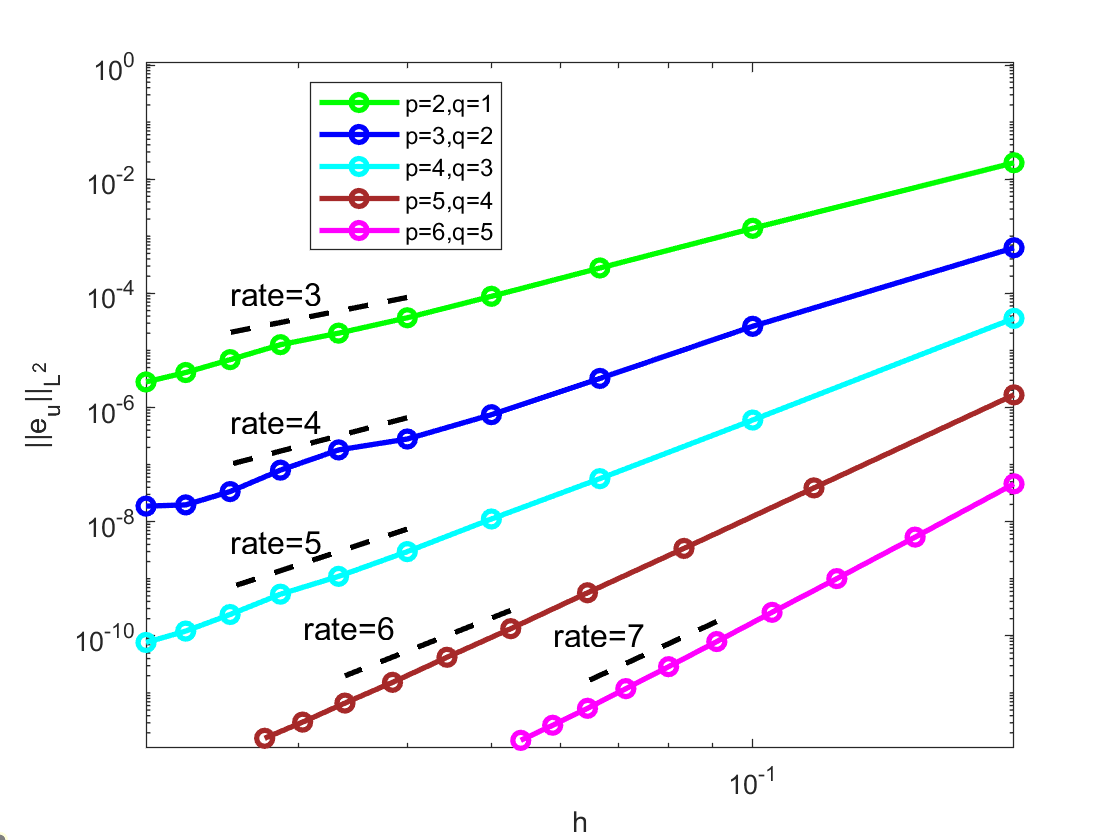}}
\subfigure[S-flux.]{\includegraphics[width=3.0in]{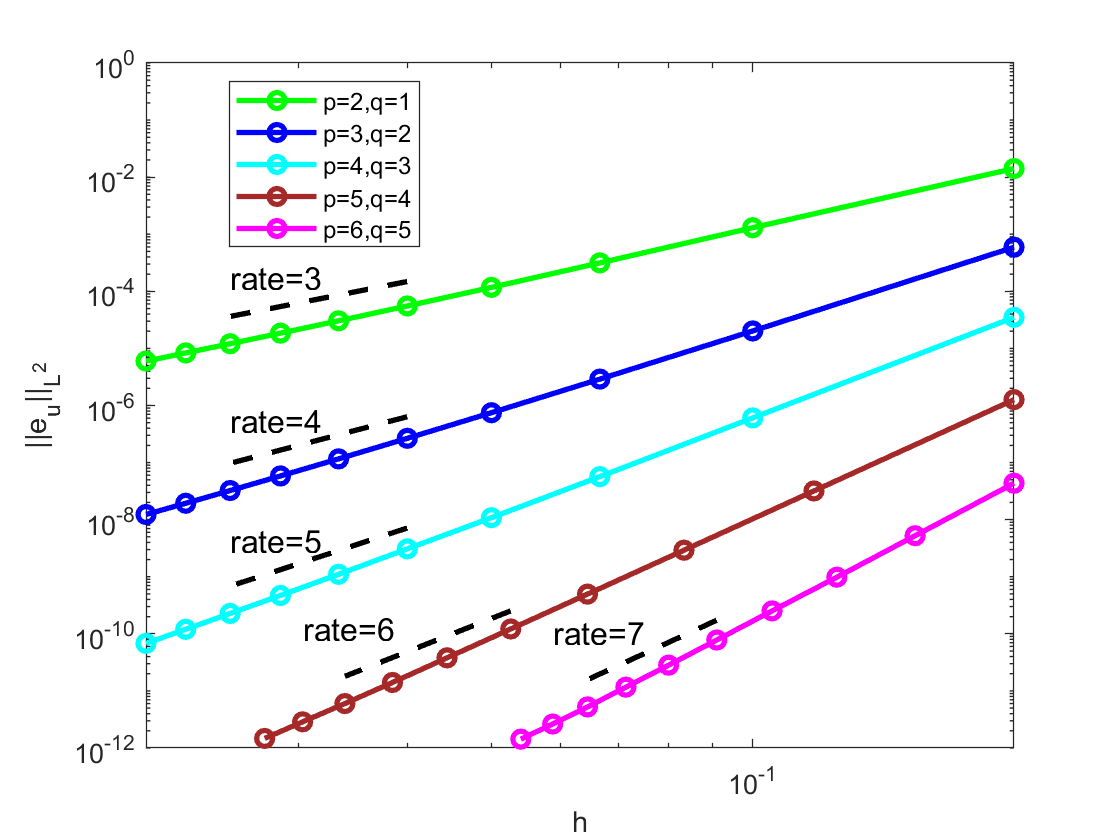}}
\subfigure[C-flux.]{\includegraphics[width=3.0in]{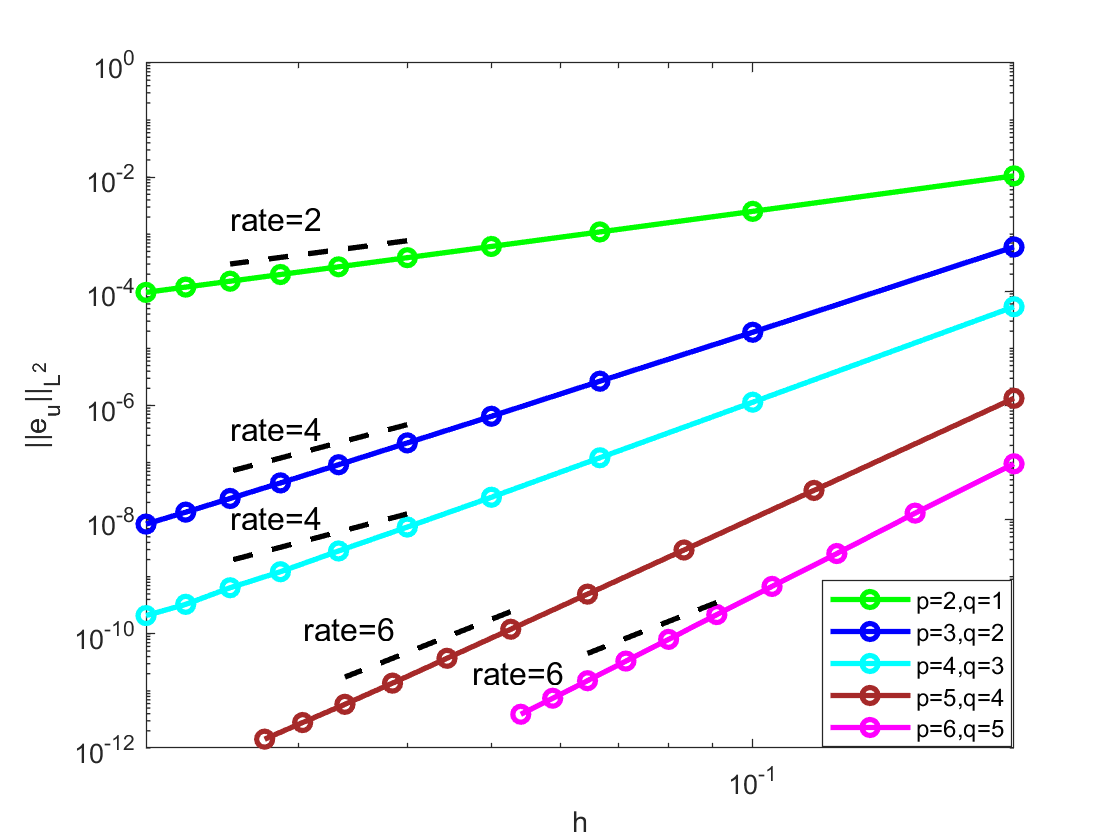}}
\caption{Example \ref{ex1}: $L^2$ errors with different numerical fluxes on nonuniform meshes.}
\label{fig:ex1_nonuniform}
\end{figure}

\begin{example}\label{ex2} (Accuracy test for nonlinear problem)\end{example}

For the case of equation \eqref{eq:wave_1D} with nonlinear source term $g(u)$, we consider soliton solutions of the Sine-Gorden equation in one dimension. Such equations appear in many physics applications and are known for their soliton and multi-soliton solutions. Here, we take the breather soliton solution as an example. In numerical tests, we choose $g(u) = -\sin u$ on domain $\Omega = (-40,40)$, and apply the homogeneous Neumann boundary conditions with the following initial conditions,
\begin{equation*}
    \begin{cases}
        u_{tt} = u_{xx} -\sin u, \\
        u(x,0) = 4\arctan\dfrac{\sqrt{0.75}}{0.5\cosh{\sqrt{0.75}x}},\\
        u_t(x,0) = 0,  \\
        u_x(-40,t) = u_x(40,t) = 0.
    \end{cases}
\end{equation*}
These conditions correspond to an exact standing breather soliton solution 
$$u(x,t) = 4\arctan\dfrac{\sqrt{0.75}\cos(0.5t)}{0.5\cosh(\sqrt{0.75}x)}.$$

In the experiments, non-linear terms are treated using the scheme (\ref{eq:DGnonlinear_1D}) with the parameter $\chi = 1$. In Figure \ref{fig:ex2}, we show the $L^2$ error of $u_h$ at the final time for different choices of numerical fluxes and polynomial degrees. We observe the optimal convergence rate with either S-flux or A-flux for all cases. For C-flux, optimal convergence rate is obtained with odd $p$, whereas the suboptimal convergence rate reduced by one order is obtained with even $p$. 

\begin{figure}[!ht] 
\centering
\subfigure[A-flux.]{\includegraphics[width=3.0in]{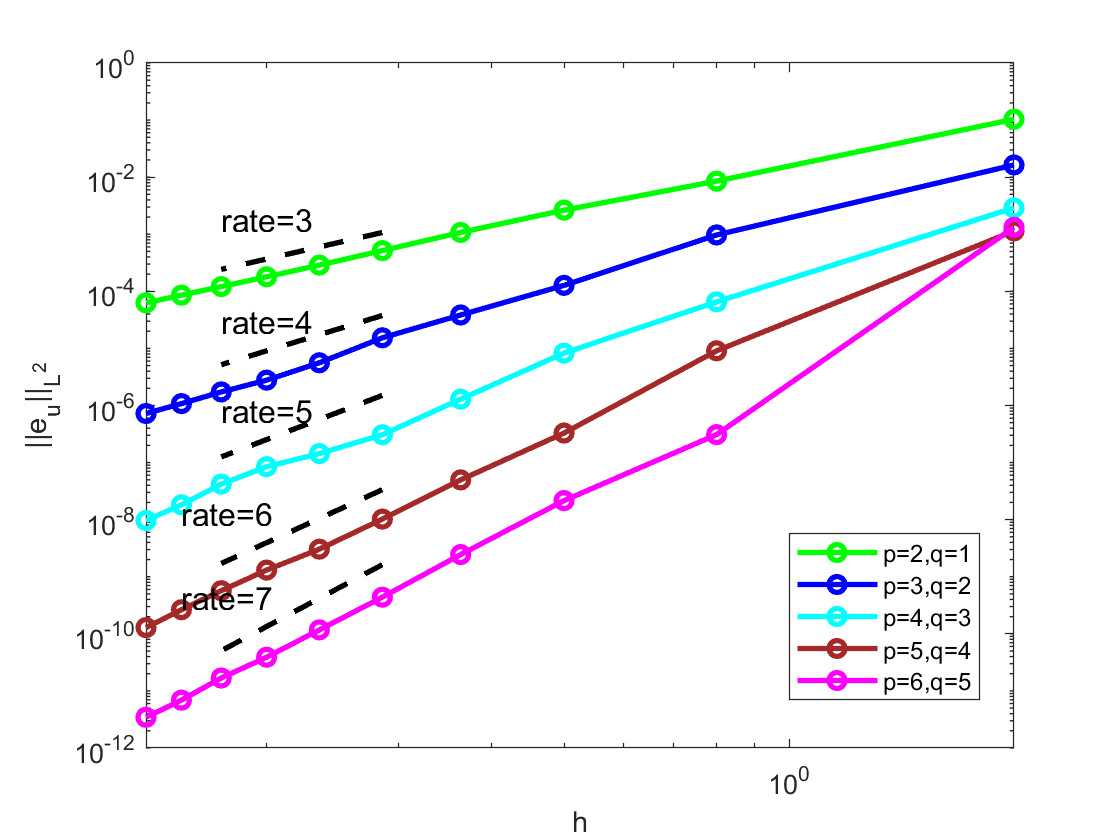}}
\subfigure[S-flux.]{\includegraphics[width=3.0in]{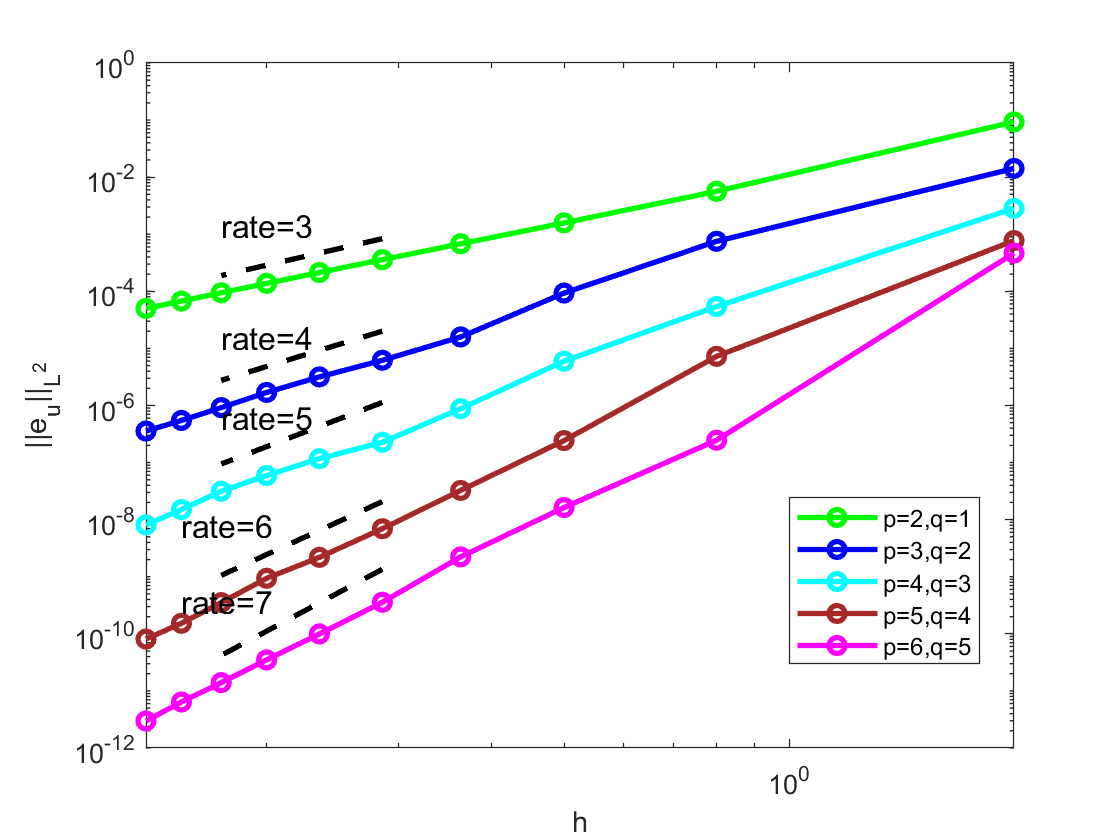}}
\subfigure[C-flux.]{\includegraphics[width=3.0in]{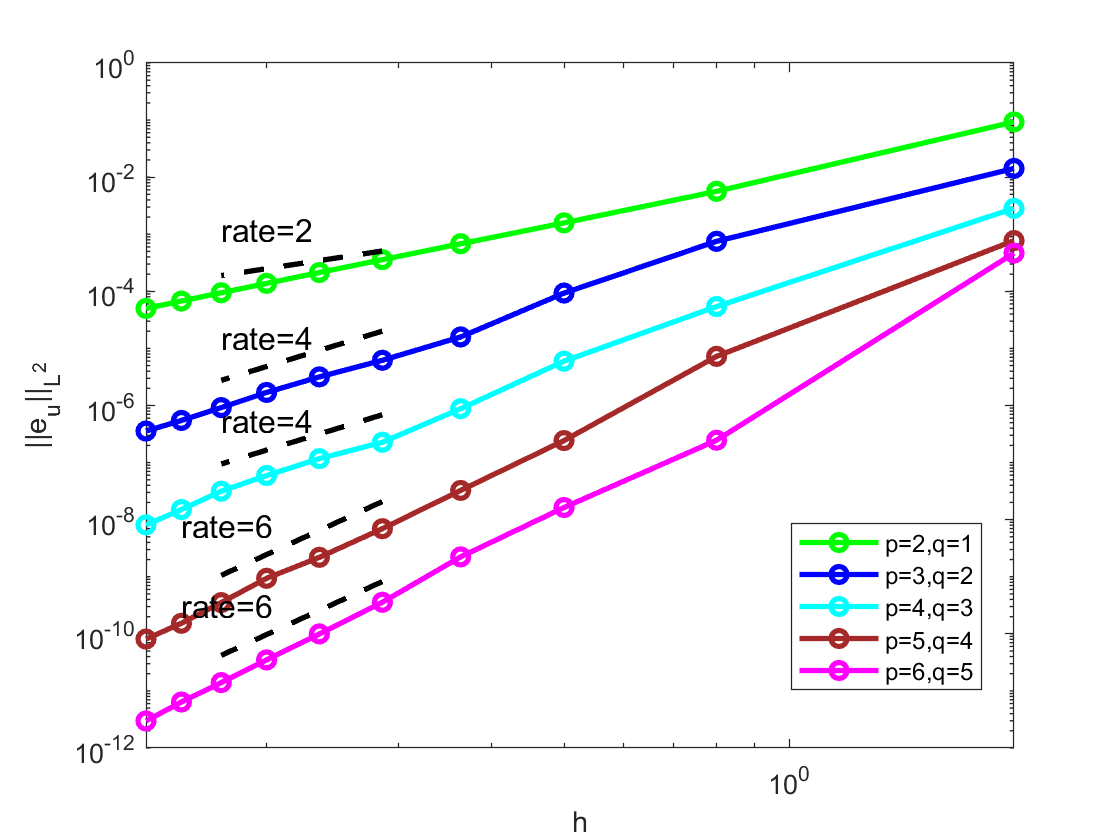}}
\caption{Example \ref{ex2}: $L^2$ errors with different numerical fluxes.}
\label{fig:ex2}
\end{figure}

\begin{example}\label{ex3} (Non-smooth solutions for linear problem)\end{example}
    
For the problems with non-smooth solutions, we consider equation \eqref{eq:wave_1D} in domain $\Omega = (-1,1)$, whose initial condition is a piecewise constant function,
    \begin{equation*}
        u(x,0) = \begin{cases}
              1.0, \quad|x|<0.5,\\
              0.5, \quad \text{otherwise,}
        \end{cases}\quad 
        u_t(x,0) = 0.
\end{equation*}

To illustrate the effectiveness of the proposed OF-EDG scheme, we choose $p = 2$, $q = 1$, and compare with the original EDG scheme, the EDG scheme with penalty term, and the EDG scheme with damping terms. We plot the numerical solutions on two different meshes consisting of 160 and 320 cells, respectively. In both cases, the discontinuities in the initial data align with the cell interfaces. From Figure \ref{fig:ex3}, we can see that the numerical solutions of the original EDG method or with the additional damping terms remain unchanged over time, generating wrong the discontinuous solutions. After adding the penalty term to the original EDG method, the numerical solution converges to the exact solution but spurious oscillations appear near the discontinuities. 
Meanwhile, the proposed OF-EDG method provides solutions in very good agreement with the exact solution.
This example demonstrates the necessity of both penalty term and damping term.

\begin{figure}[!htbp]
\centering
\subfigure[EDG with 160 cells.]{\includegraphics[width=0.5\textwidth, height=1.7in, keepaspectratio]{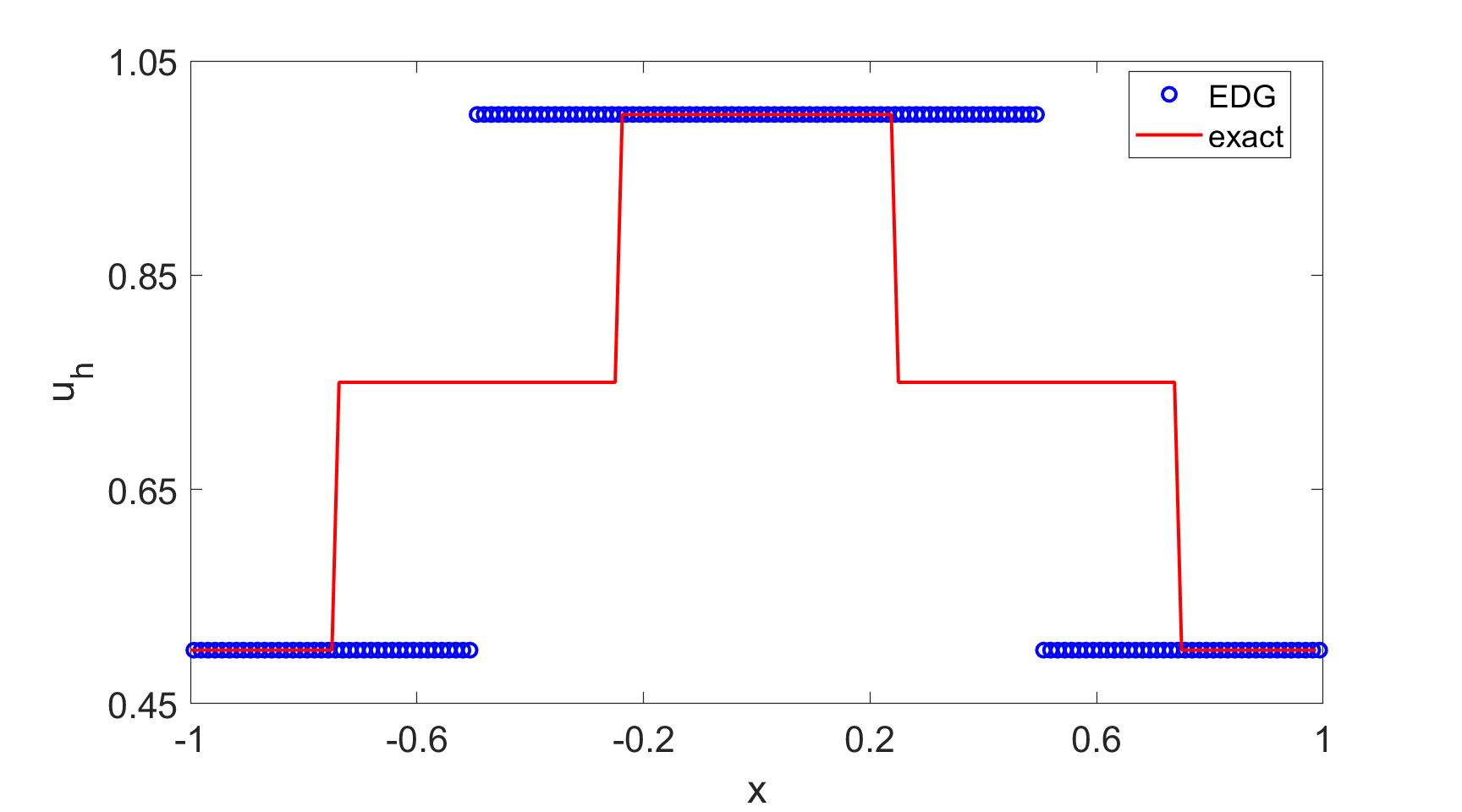}}
\subfigure[EDG with 320 cells.]{\includegraphics[width=0.5\textwidth, height=1.7in, keepaspectratio]{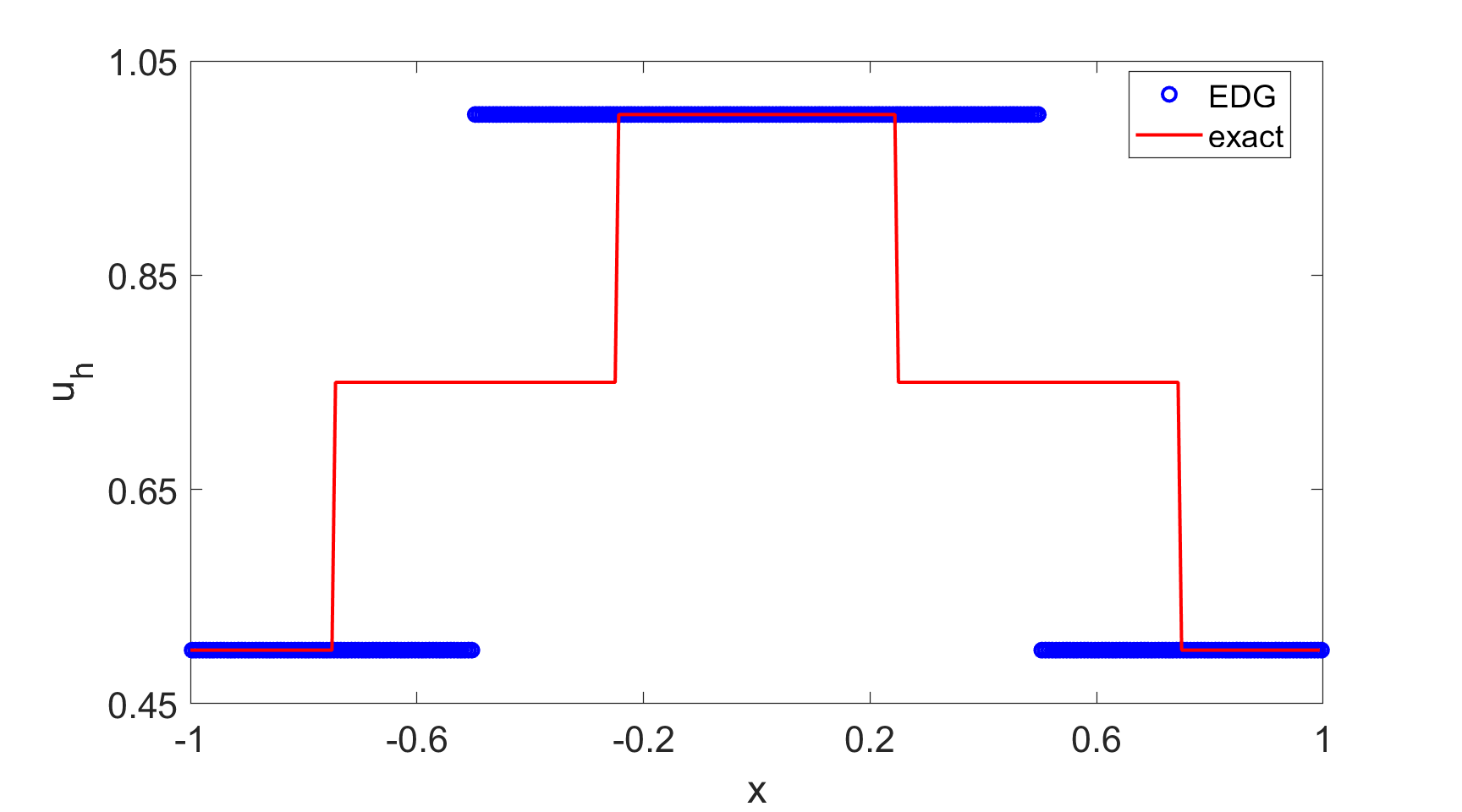}}
\subfigure[EDG with damping terms and 160 cells.]{\includegraphics[width=0.5\textwidth, height=1.7in, keepaspectratio]{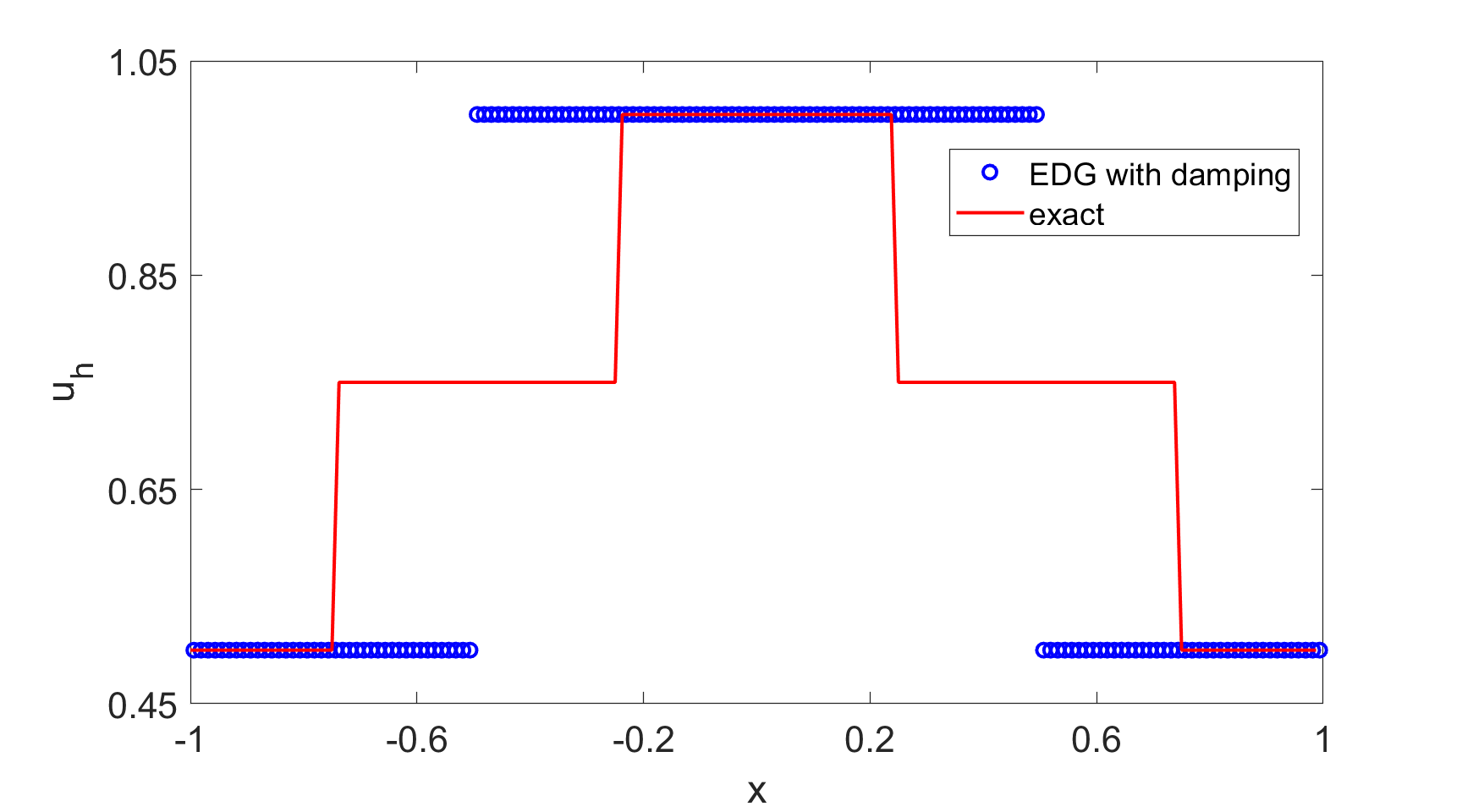}}
\subfigure[EDG with damping terms and 320 cells.]{\includegraphics[width=0.5\textwidth, height=1.7in, keepaspectratio]{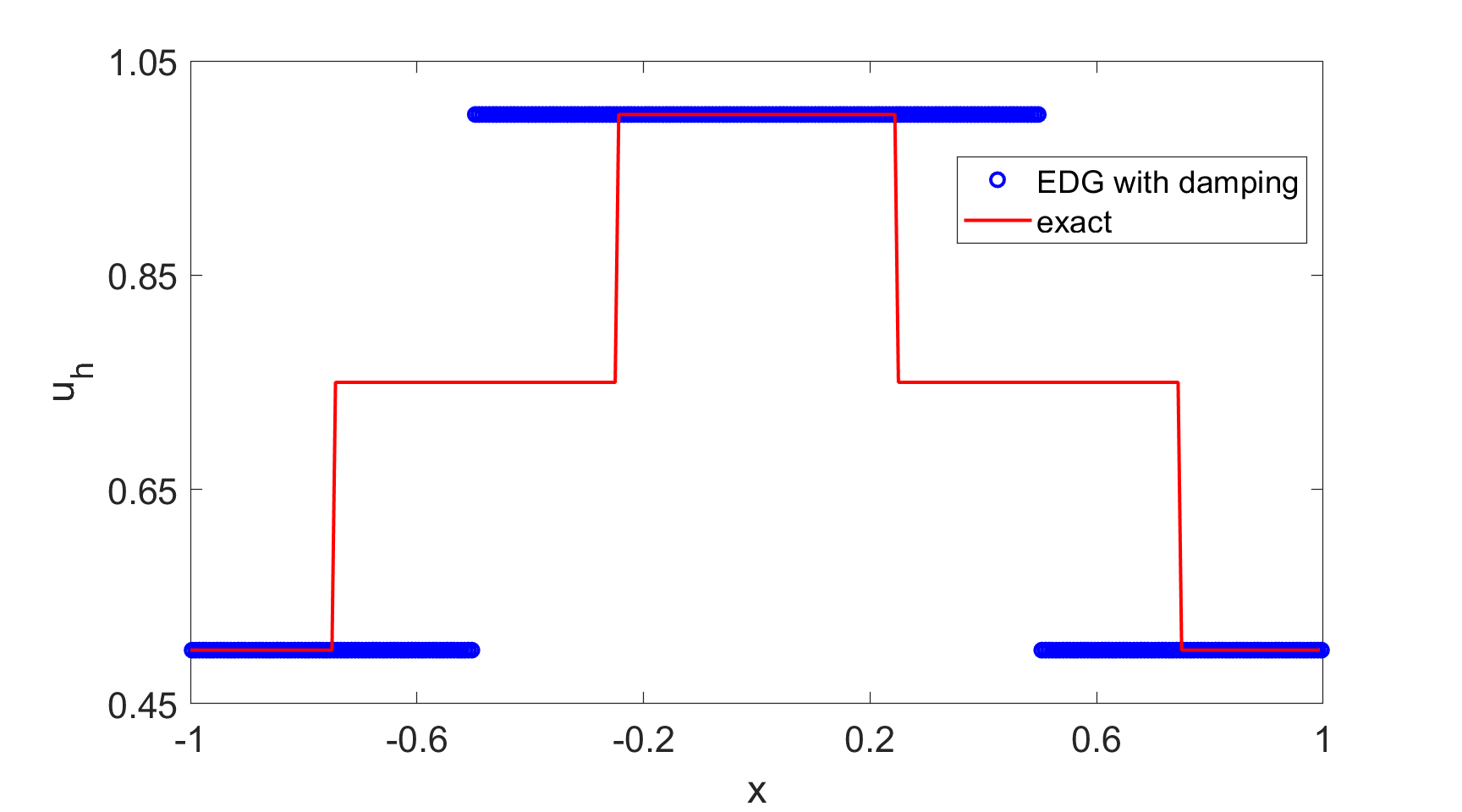}}
\subfigure[EDG with penalty term and 160 cells.]{\includegraphics[width=0.5\textwidth, height=1.7in, keepaspectratio]{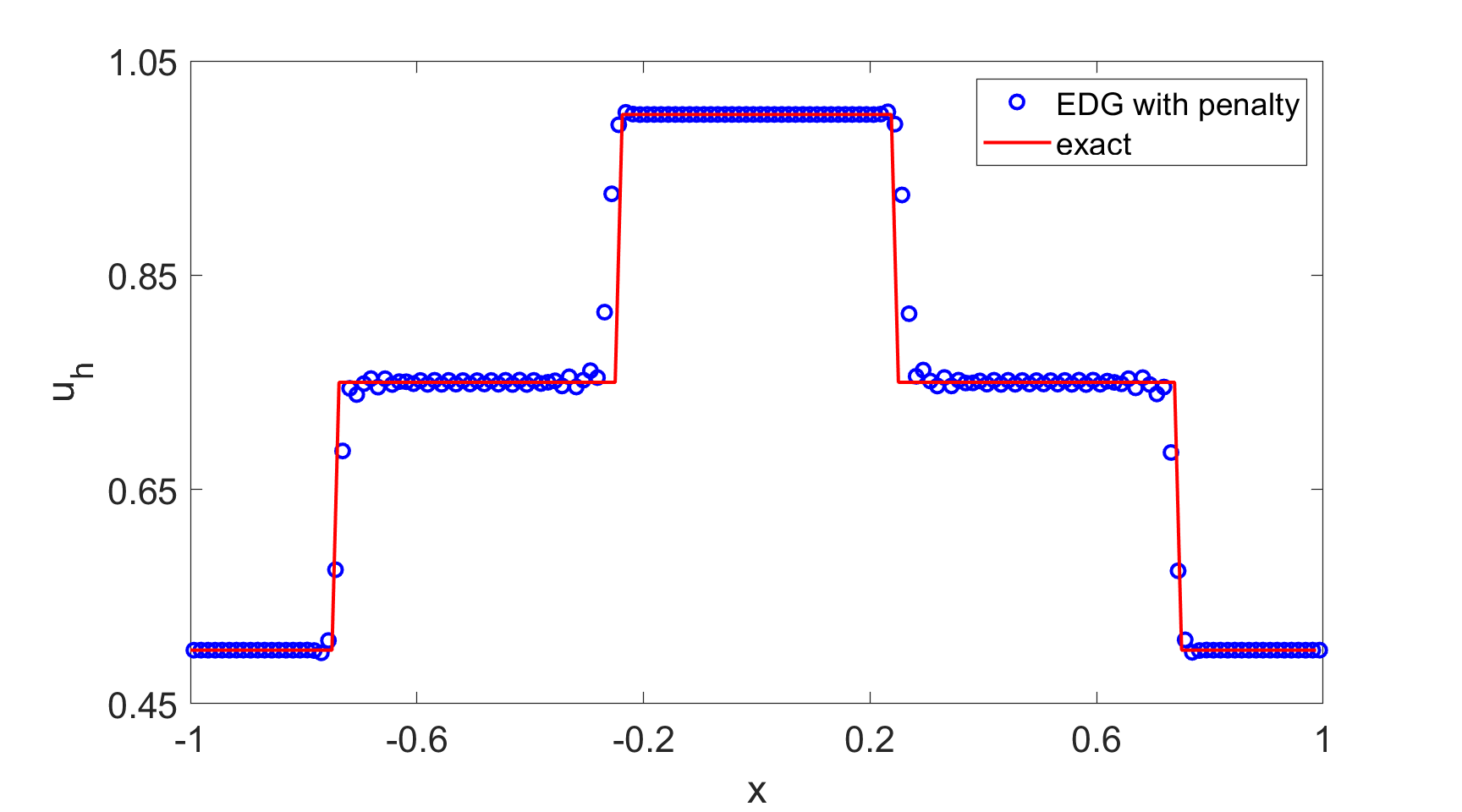}}
\subfigure[EDG with penalty term and 320 cells.]{\includegraphics[width=0.5\textwidth, height=1.7in, keepaspectratio]{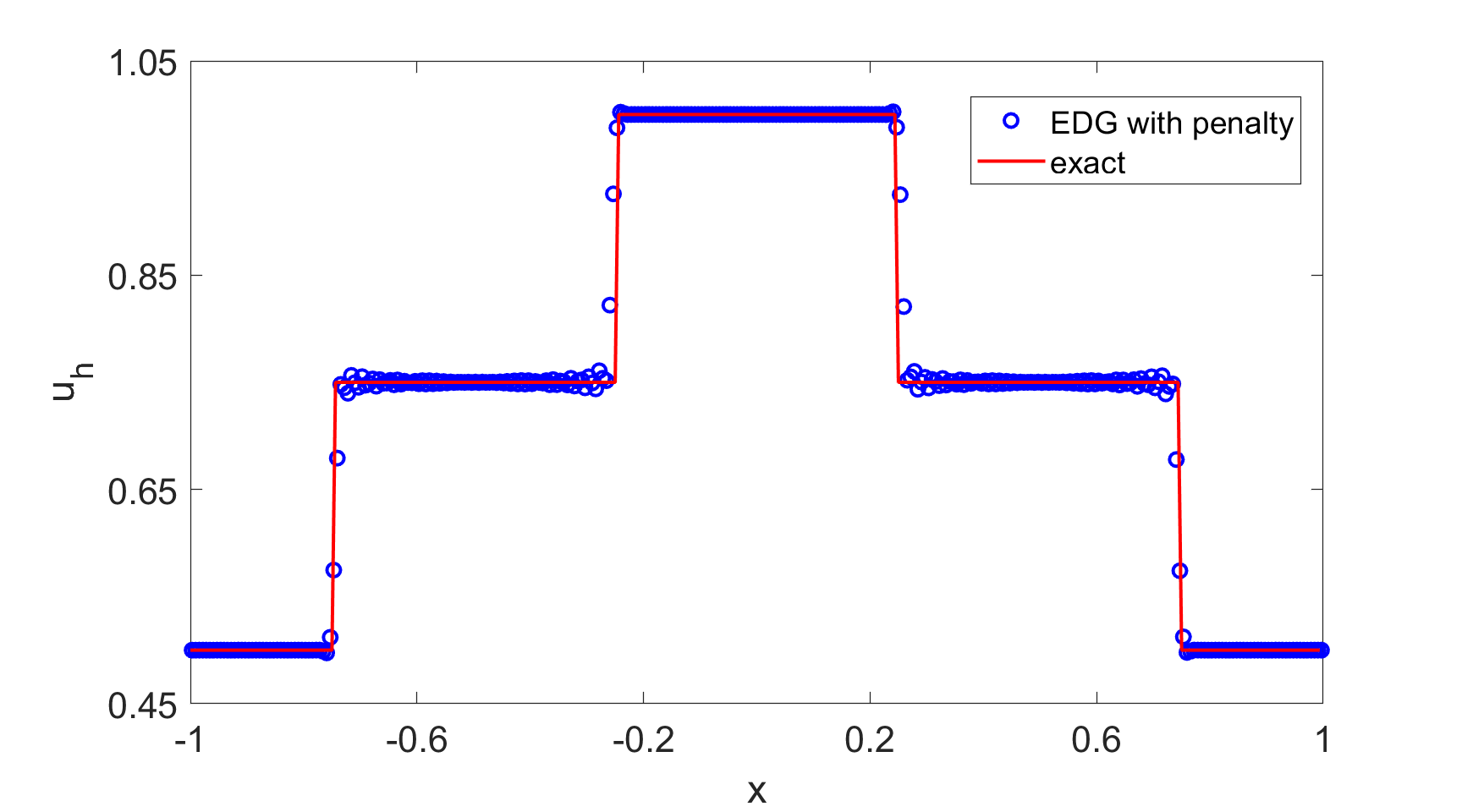}}
\subfigure[OF-EDG with 160 cells.]{\includegraphics[width=0.5\textwidth, height=1.7in, keepaspectratio]{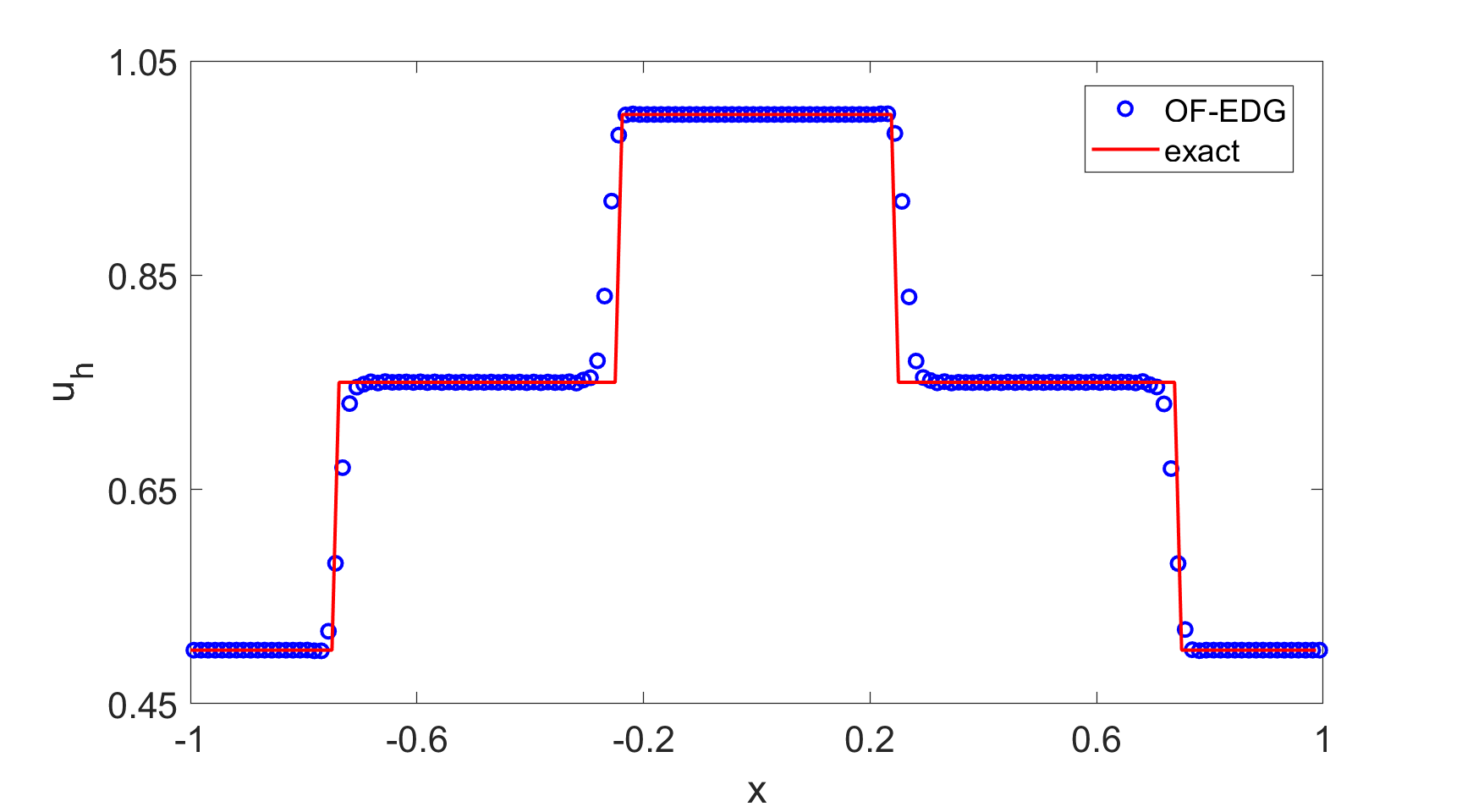}}
\subfigure[OF-EDG with 320 cells.]{\includegraphics[width=0.5\textwidth, height=1.7in, keepaspectratio]{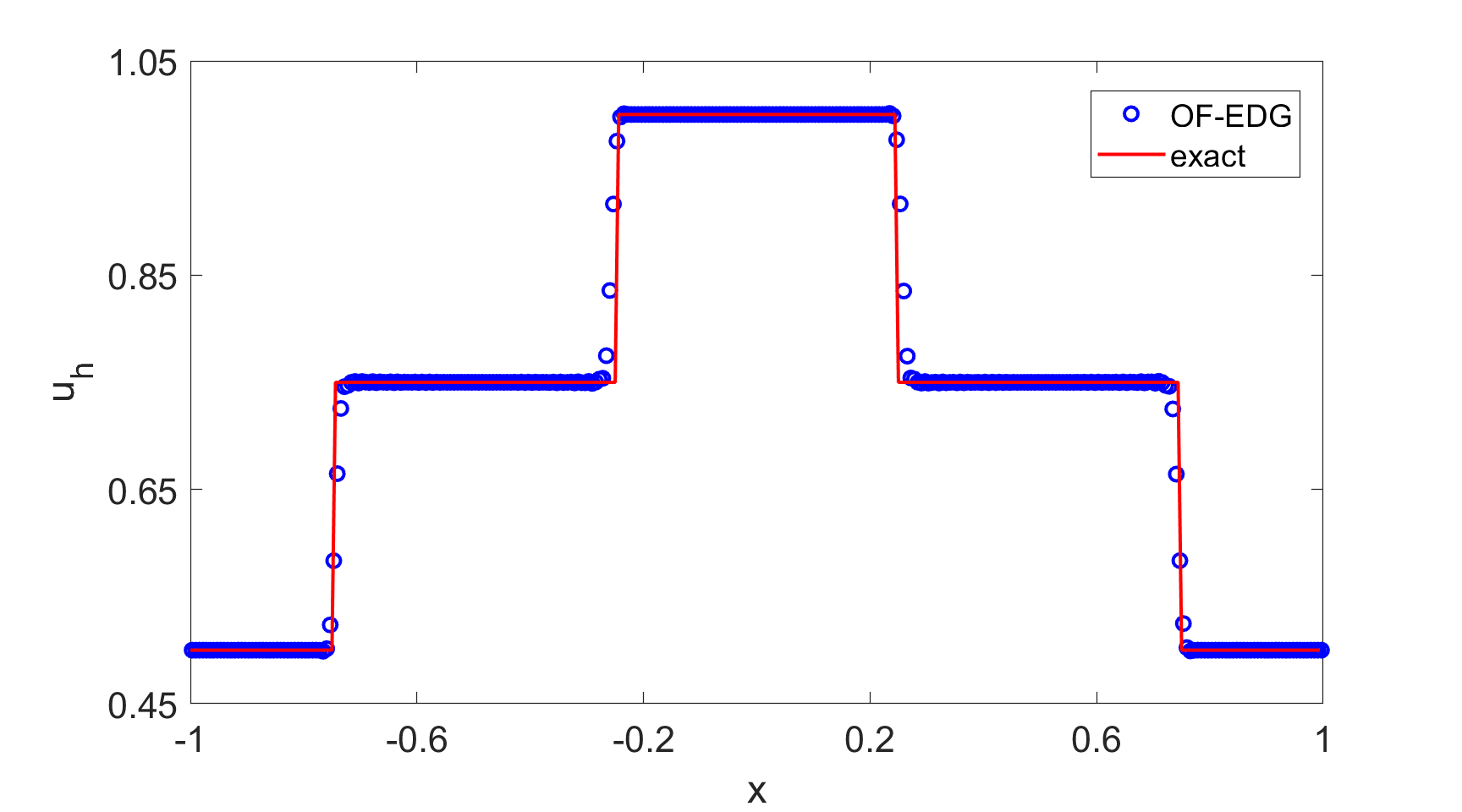}}
\caption{{Example \ref{ex3}: numerical solutions with different schemes and cell numbers at $t=0.25$. }}
\label{fig:ex3}
\end{figure}

\begin{example}\label{ex4} (Sine-Gorden equation)\end{example}

We consider the semi-linear wave equation in domain $\Omega = (0,1)$, which contains a nonlinear source term $g(u)=160\sin(u)$,      
     \begin{equation*}
         u_{tt} = u_{xx} + 160\sin(u),
     \end{equation*} 
and the initial values are piecewise constant function as follows,
\begin{equation*}
u(t,0) =
    \begin{cases}
        5, & 0.3\leq x\leq 0.425,\\
        2.5,& 0.575\leq x \leq 0.7,\\
        0,& \text{otherwise},
    \end{cases}
    \quad u_t(x,0) = 0.
\end{equation*}
In this experiment, we employ scheme (\ref{eq:DGnonlinear_1D}) with $\chi = 1$ to treat nonlinear terms. We choose $p = 2$, $q = 1$. In this example, an exact solution is not available. Instead, we compare the numerical solution obtained using our OF-EDG method on a mesh with 320 cells to the solution via a standard finite difference Central in Time and Central in Space (CTCS) method on a finer mesh with 1000 points. As shown in Figure \ref{fig:ex4}, we  observe that both solutions exhibit the same overall shape when $t = 0.25$. However, the OF-EDG solution is free from oscillations, highlighting its superior stability and accuracy near discontinuities.

\begin{figure}[!ht]
\centering
\subfigure[CTCS scheme.]{\includegraphics[width=3.0in]{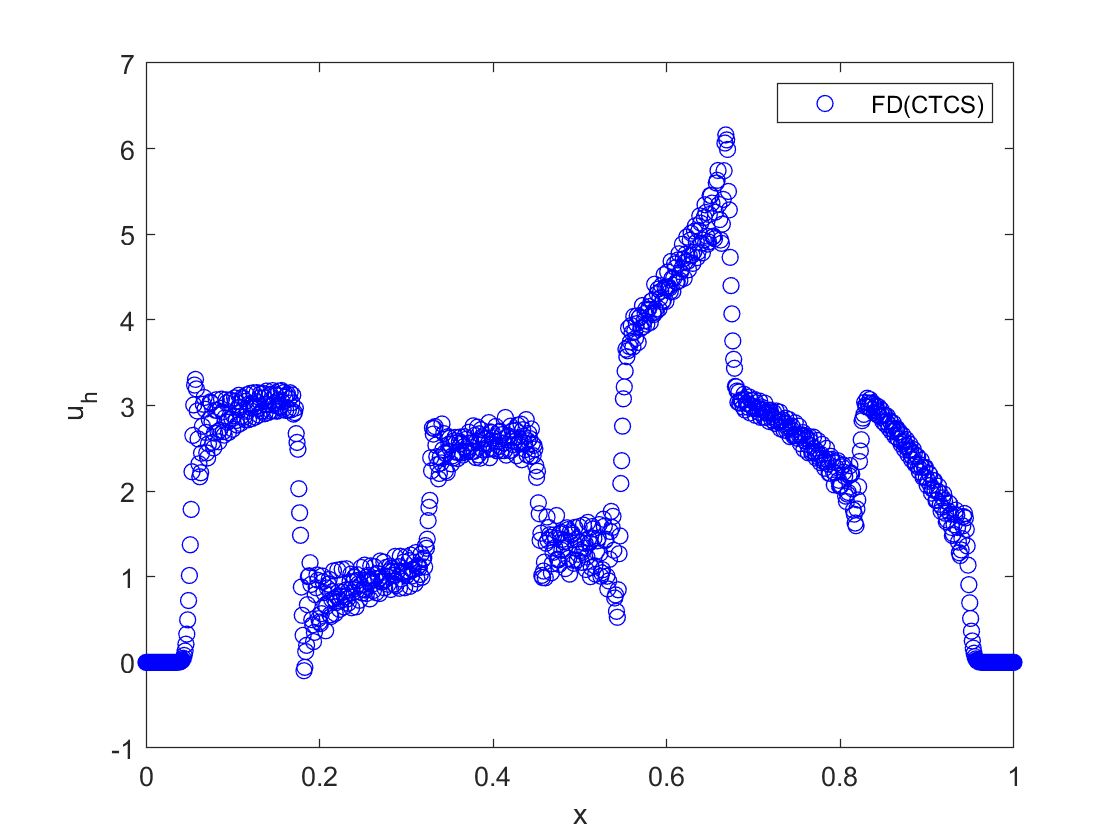}}
\subfigure[OF-EDG scheme.]{\includegraphics[width=3.0in]{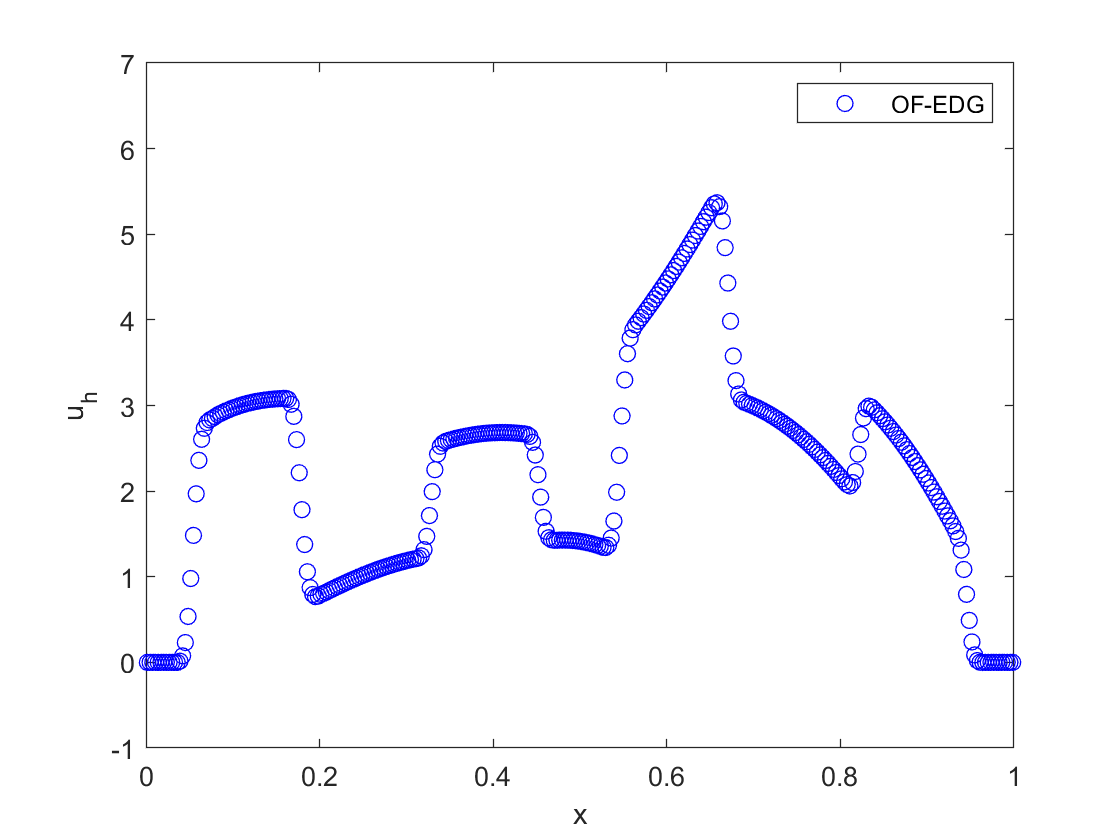}}
\caption{{Example \ref{ex4}: numerical results of $u_h$ at $t=0.25$.}}
\label{fig:ex4}
\end{figure}

\begin{example}\label{ex5} (Klein-Gorden equation) \end{example}
    
We consider the semi-linear wave equation in domain $\Omega = (0,1)$, which contains a nonlinear source term $g(u)=4u^3$, 
     \begin{equation*}
         u_{tt} = u_{xx} + 4u^3,
     \end{equation*} 
 and piecewise constant initial data,
\begin{equation*}
    u(t,0) =
    \begin{cases}
        4, & 0.3\leq x\leq 0.425,\\
        2, & 0.575\leq x \leq 0.7,\\
        0, & \text{otherwise}.
    \end{cases}
    \qquad u_t(x,0) = 0,
\end{equation*}
 and final time $t = 0.25$. We use scheme (\ref{eq:DGnonlinear_1D}) with $\chi = 1$ to treat the nonlinear term. We choose $p = 2, q = 1$. Since the exact solution is not available, we compare the numerical solution obtained using our OF-EDG method on a mesh with 320 cells to the solution computed using CTCS method on a finer mesh with 1000 cells. 
From Figure \ref{fig:ex5}, it can be seen that both solutions exhibit the same overall shape when $t = 0.25$, and the OF-EDG solution is free from oscillations. 

\begin{figure}[!ht]
\centering
\subfigure[CTCS scheme.]{\includegraphics[width=3.0in]{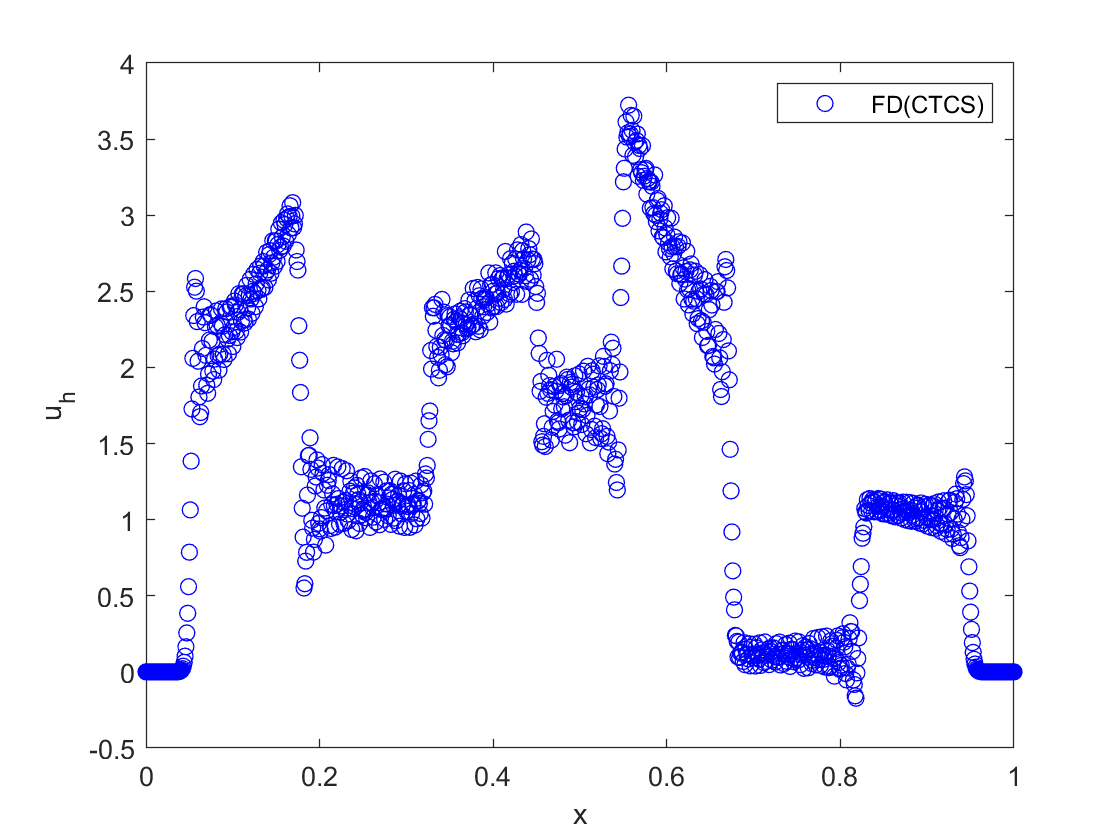}}
\subfigure[OF-EDG scheme.]{\includegraphics[width=3.0in]{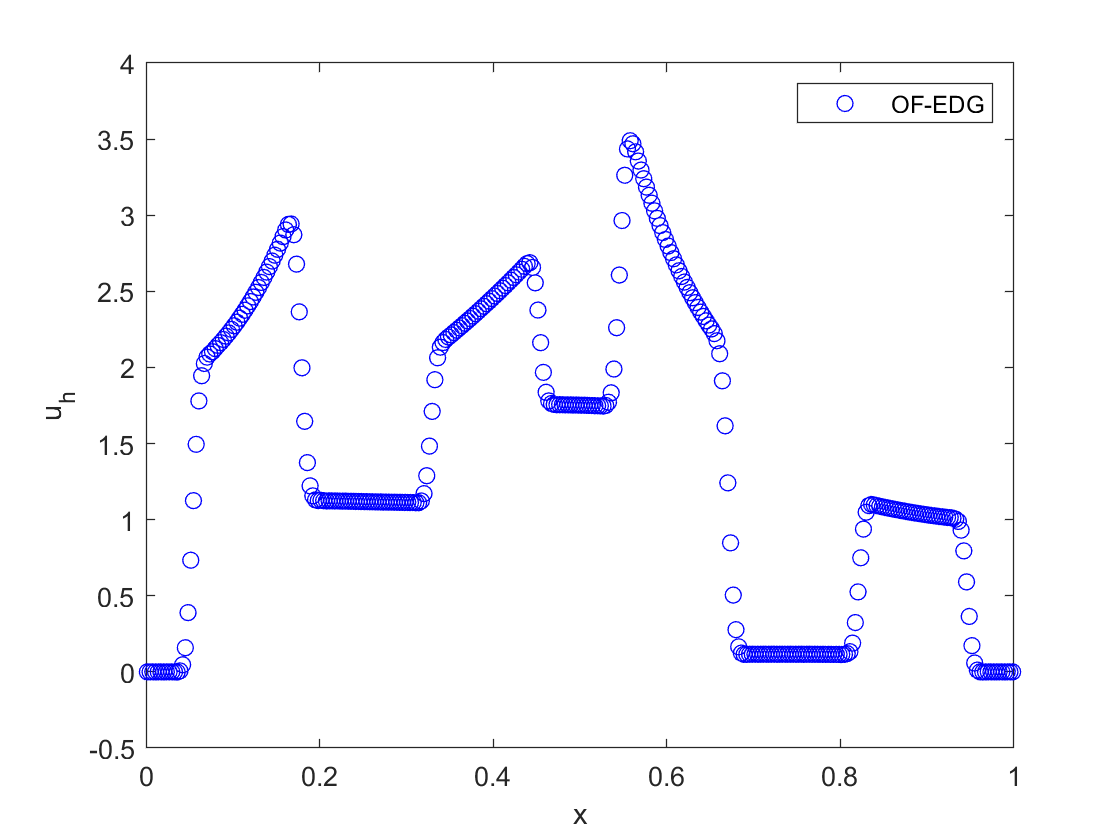}}
\caption{{Example \ref{ex5}: numerical results of $u_h$ at $t=0.25$.}}
\label{fig:ex5}
\end{figure}

\begin{remark}
We have also solved the nonlinear problems in Example \ref{ex4}-\ref{ex5} using scheme \eqref{eq:DGnonlinear_1D} with $\chi=0$, and observed almost identical results as $\chi=1$. For problems with discontinuous solutions, we have found that the scheme with $\chi=1$ is sensitive to the damping and penalty terms and requires careful adjustment on those parameters to achieve the same non-oscillatory effect.
\end{remark}


\subsection{Two dimensional problems}

\begin{example}\label{ex6} (Accuracy test for 2D linear problem)\end{example}
In this example, we consider the two-dimensional wave equation \eqref{eq:wave_2D} in domain $\Omega = [-\pi,\pi]^2$ with exact solution $u(x,t) = \sin(x+y+\sqrt{2}t)$. We test the problem with $p = 2,3,4$ and $q=p-1$. 
In Figure \ref{fig:ex6}, the $L^2$ errors of $u_h$ on uniform meshes at the final time $t=0.25$ for different choices of numerical fluxes and polynomial degrees are presented,  
demonstrating the same conclusion as the 1D case in the least-squares sense.

\begin{figure}[!ht]
\centering
\subfigure[A-flux.]{\includegraphics[width=3.0in]{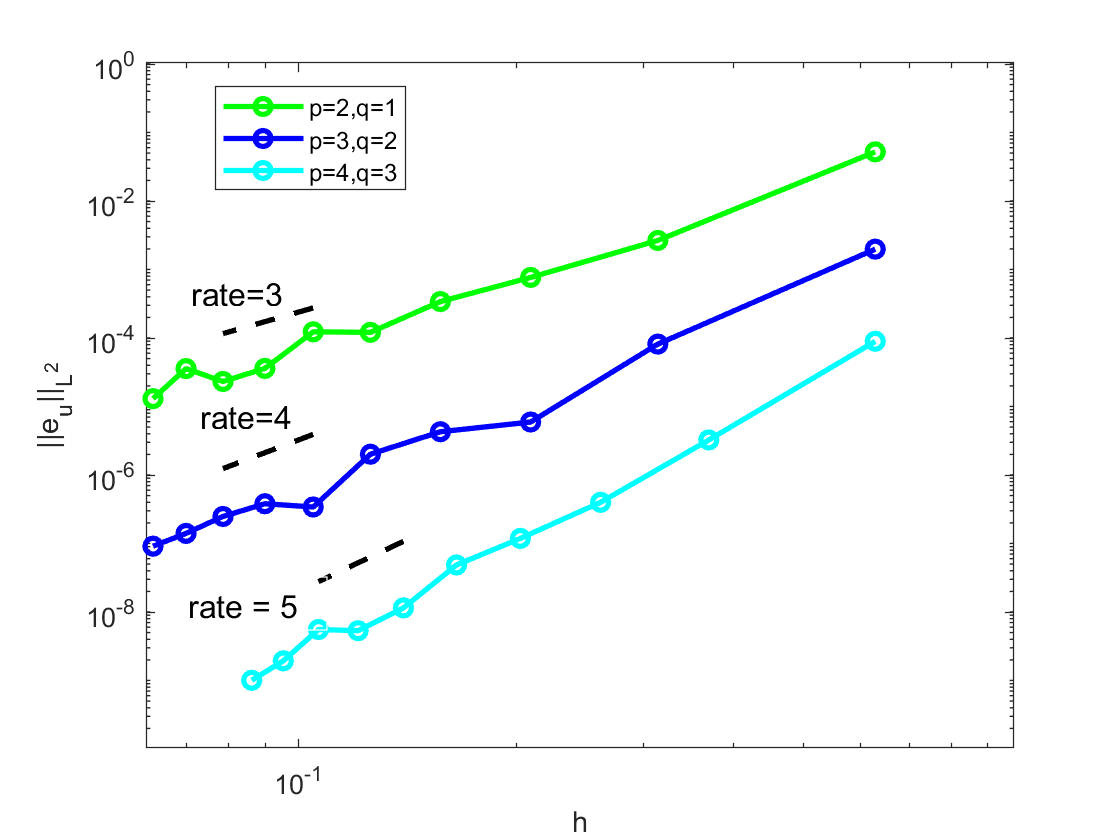}}
\subfigure[S-flux.]{\includegraphics[width=3.0in]{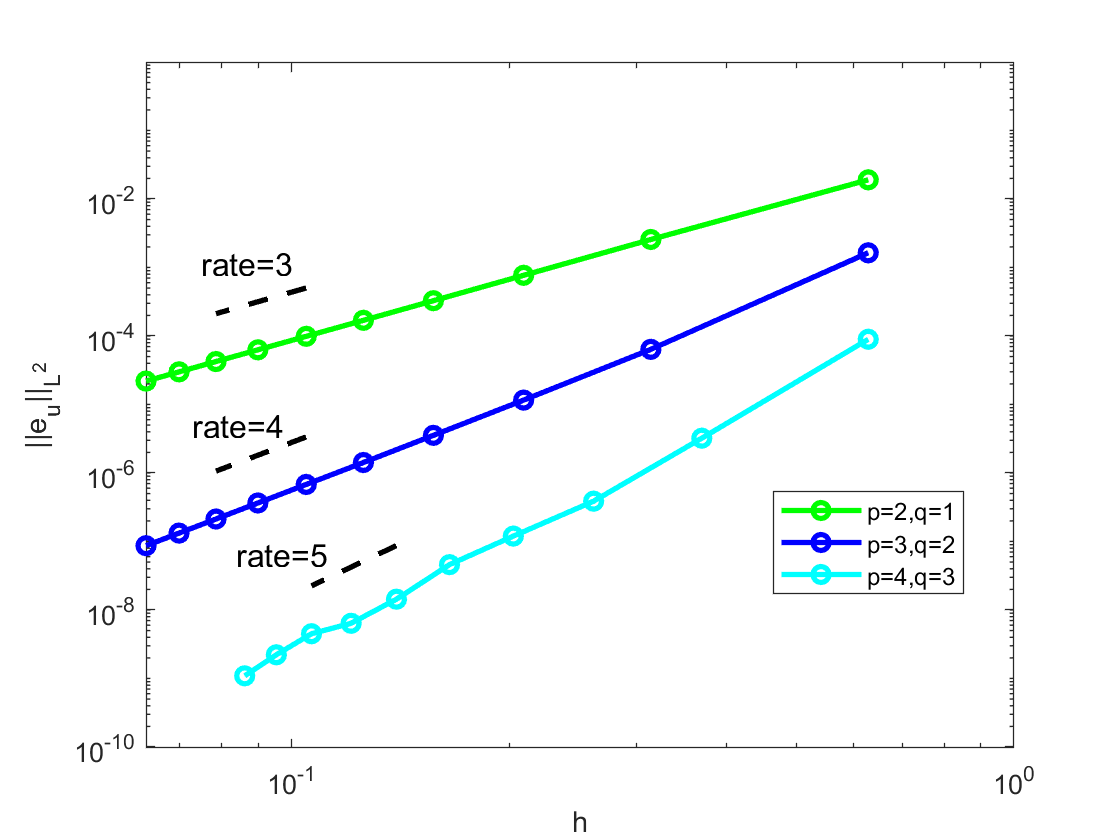}}
\subfigure[C-flux.]{\includegraphics[width=3.0in]{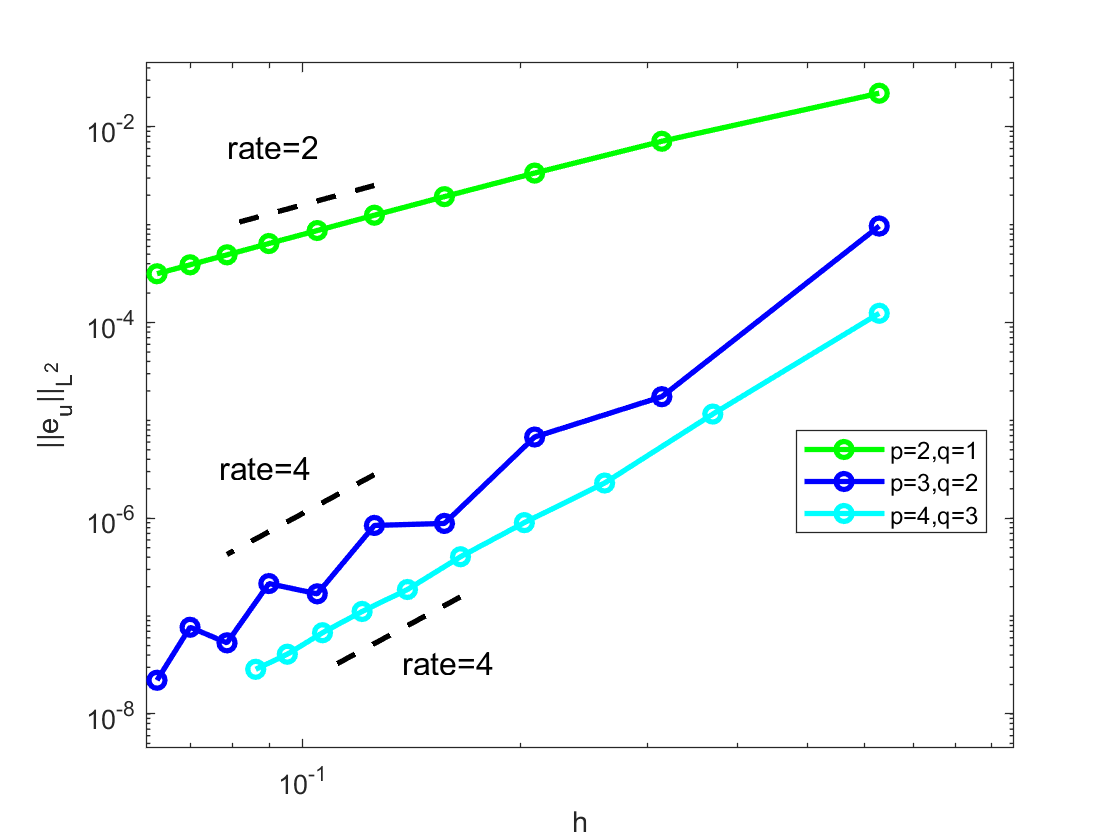}}
\caption{Example \ref{ex6}: $L^2$ errors of $u_h$ with different numerical fluxes.}
\label{fig:ex6}
\end{figure}

\begin{example}\label{ex7} (Sine-Gorden equation in 2D)\end{example}

We consider the 2D semi-linear Sine-Gorden equation in domain $\Omega = [-1,1]^2$, 
     \begin{equation*}
         u_{tt} = u_{xx} + u_{yy} + 16\sin(u),
     \end{equation*} 
with the piecewise constant initial values
\begin{equation*}
    u(x,y,0) = \begin{cases}
        0.5, & (x,y)\in [0.375,0.625]^2,\\
        0,& \text{otherwise},
    \end{cases}
    \qquad u_t(x,y,0) = 0.
\end{equation*}
Here, we choose $p = 2, q = 1$ and final time $t = 0.25$. In Figure \ref{fig:ex7}, we plot the numerical solution using our OF-EDG method on a mesh with $200\times200$ cells and that using a standard CTCS method on with $1000\times1000$ cells. It can be seen that the OF-EDG solution is free from oscillations, highlighting its superior stability near discontinuities.

\begin{figure}[!ht]
\centering
\subfigure[CTCS scheme.]{\includegraphics[width=3.0in]{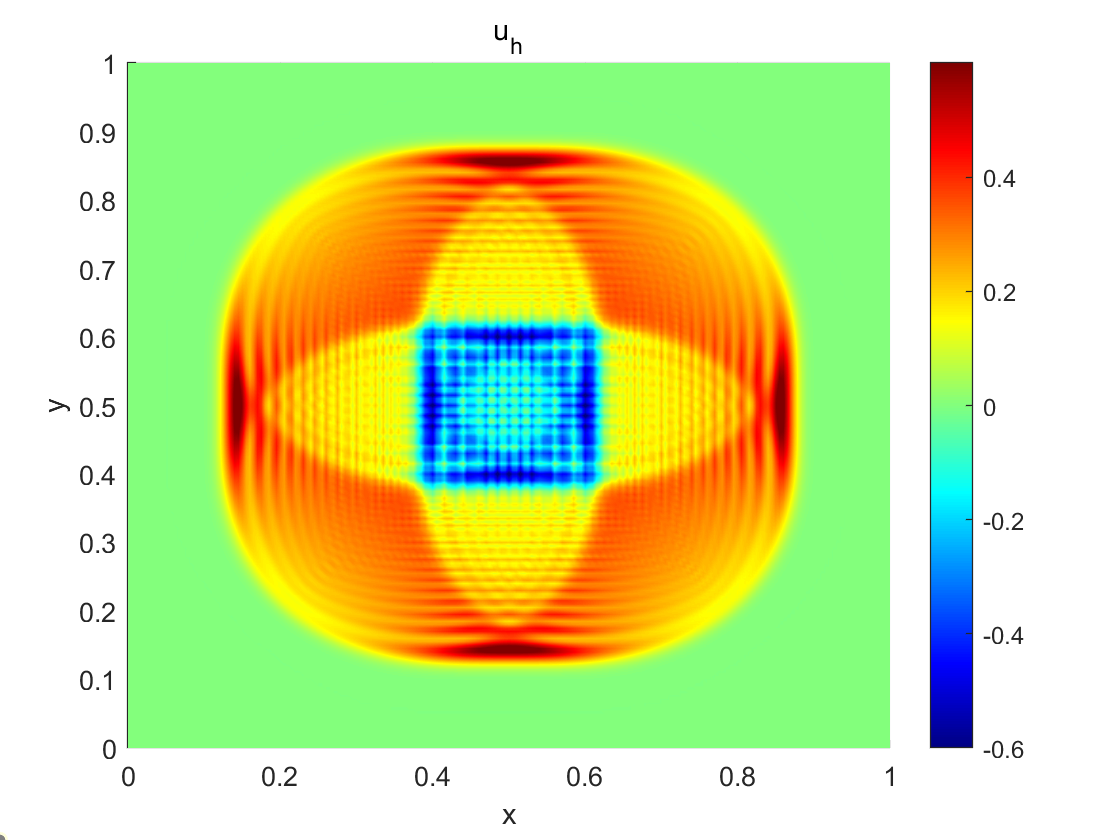}}
\subfigure[OF-EDG scheme.]{\includegraphics[width=3.0in]{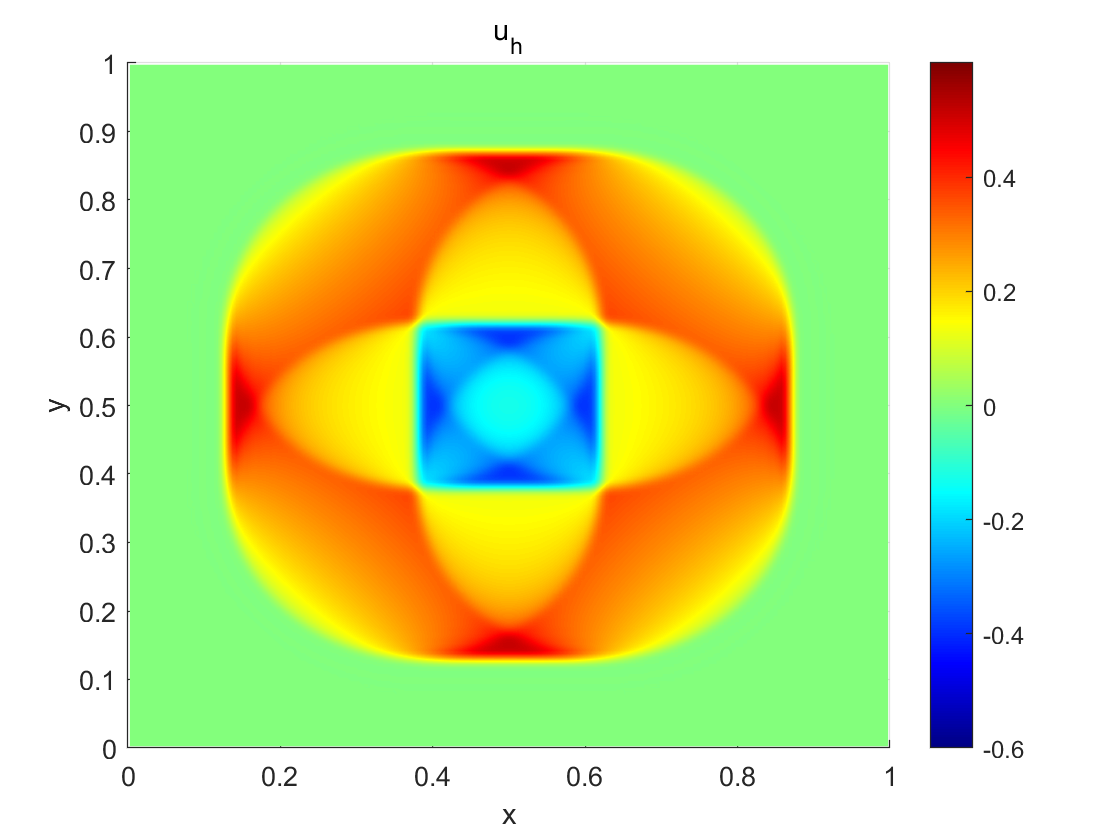}}
\caption{Example \ref{ex7}: numerical solutions $u_h$ at $t=0.25$.}
\label{fig:ex7}
\end{figure}

\begin{example}\label{ex8} (Klein-Gorden equation in 2D)\end{example}
   
Finally, we consider the 2D semi-linear Klein-Gorden equation on $\Omega = [-1,1]^2$, 
     \begin{equation*}
         u_{tt} = u_{xx} + 4u^3.
     \end{equation*} 
with initial values 
\begin{equation*}
    u(x,y,0) = \begin{cases}
        0.5, & (x,y)\in [0.3,0.425]^2,\\
        0.25,& (x,y)\in [0.575,0.7]^2,\\
        0,& \text{otherwise},
    \end{cases}
    \qquad u_t(x,y,0) = 0.
\end{equation*}
Again, we choose $p = 2, q = 1$ and
compare the numerical solution of OF-EDG method on a mesh with $320\times320$ cells to that of CTCS scheme with $1000\times1000$ cells at final time $t = 0.25$. Figure \ref{fig:ex8} demonstrates the advantages of our algorithm.

\begin{figure}[htpb]
\centering
\subfigure[CTCS scheme.]{\includegraphics[width=3.0in]{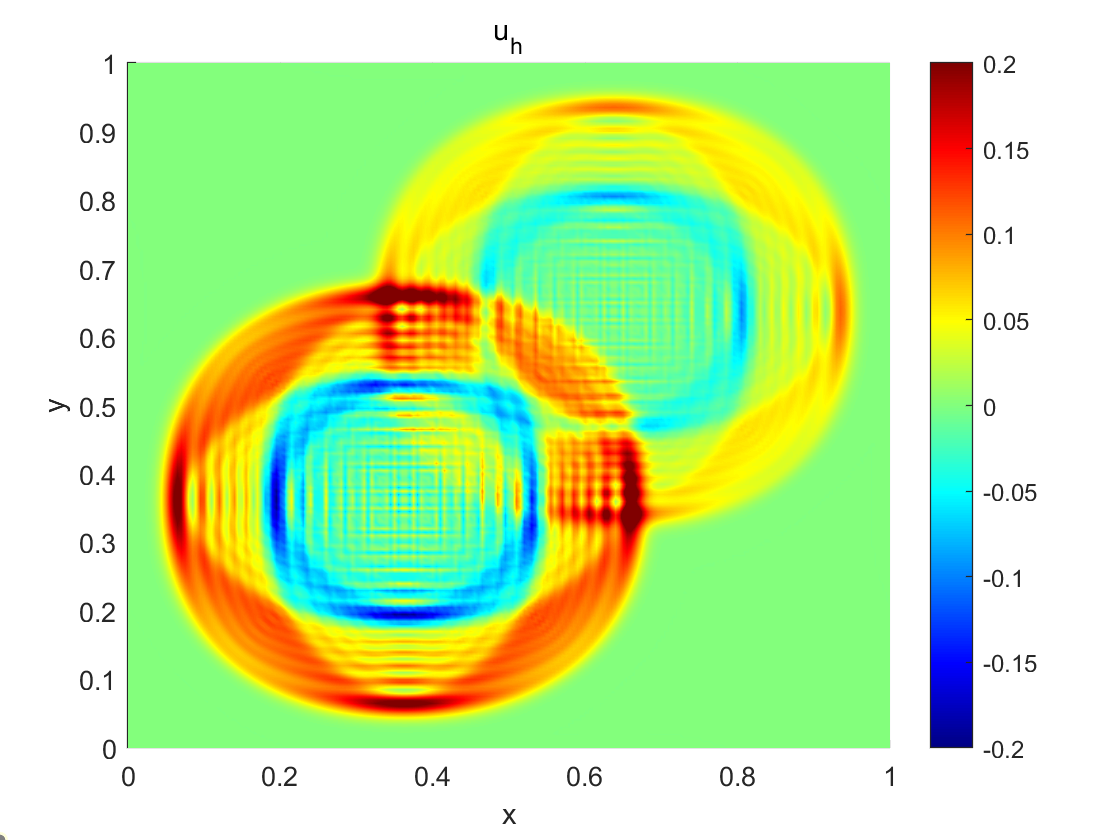}}
\subfigure[OF-EDG scheme.]{\includegraphics[width=3.0in]{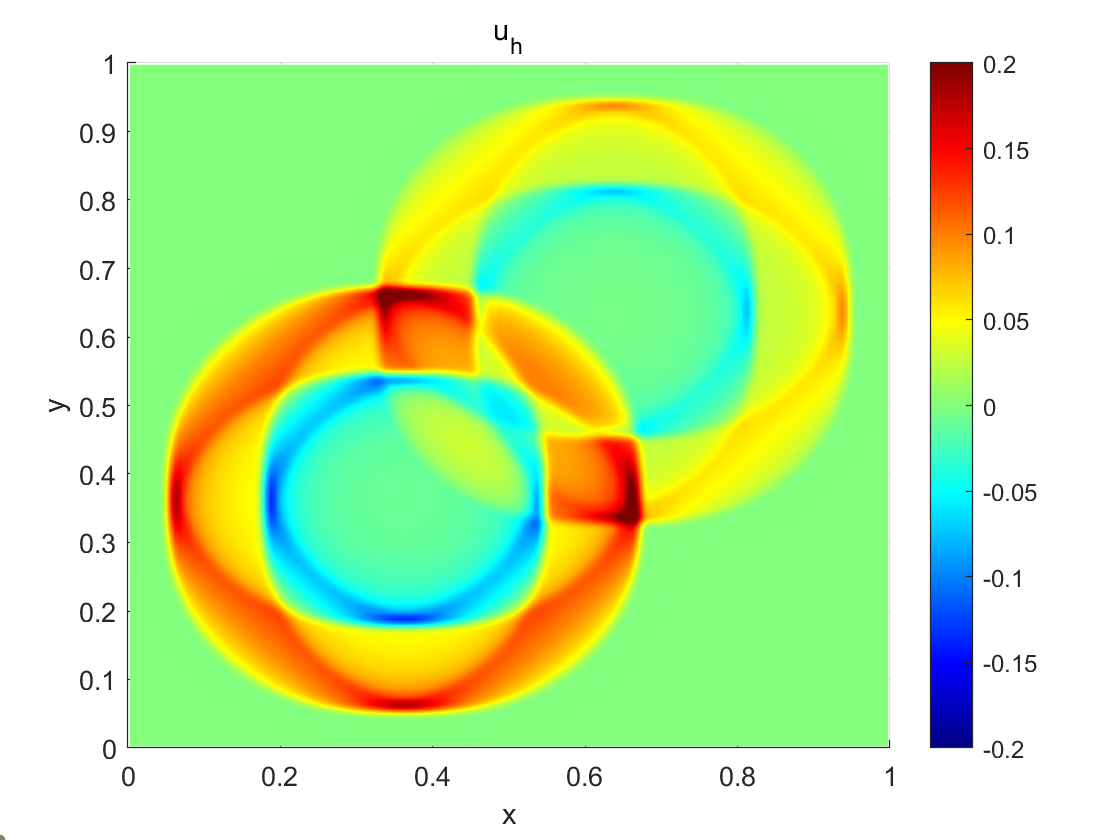}}
\caption{Example \ref{ex8}: numerical solutions $u_h$ at $t=0.25$.}
\label{fig:ex8}
\end{figure}

\section{Conclusion}\label{sec:con}
In this paper, we develop the OF-EDG method for solving the second order wave equation. Since the original EDG method is designed for smooth problems, the method produces spurious oscillations or even fails to converge when solution contains discontinuities. To overcome this difficulty, we introduce extra damping terms and penalty terms in the OF-EDG method. We prove stability and derive a priori error estimate for both one-dimensional and multi-dimensional problems. Several numerical examples are provided to verify the stability and accuracy analysis of the OF-EDG method. In addition, we have also solved the wave equation with nonlinear source terms to demonstrate the robustness of the proposed method. In the future, we would like to extend the OF-EDG method for wave propagation problems with complex geometry and nonlinear wave speed.

\section*{Acknowledgments}

The authors would like to express their gratitude to Professor Chi-Wang Shu from Brown University and Professor Yong Liu from Chinese Academy of Sciences for their valuable discussions.

\bibliographystyle{abbrv}
\bibliography{refs}

\end{document}